\documentclass[final,onefignum,onetabnum]{siamart190516}
\usepackage{lipsum}
\usepackage{amsfonts}
\usepackage{graphicx}
\usepackage{epstopdf}
\usepackage{algorithmic}
\usepackage{todonotes}
\usepackage{dsfont}
\ifpdf
  \DeclareGraphicsExtensions{.eps,.pdf,.png,.jpg}
\else
  \DeclareGraphicsExtensions{.eps}
\fi

\newsiamremark{remark}{Remark}
\newsiamremark{assumption}{Assumption} 
\crefname{hypothesis}{Hypothesis}{Hypotheses}
\newsiamthm{example}{Example}
\newtheorem{thm}{Theorem}
\headers{Regularization Theory Based on Image Space Approximation}{P. Miller}

\title{Variational Regularization Theory Based on Image Space Approximation Rates}

\author{Philip Miller\thanks{Institute for Numerical and Applied Mathematics, University of G\"ottingen
  (\email{p.miller@math.uni-goettingen.de}}).}

\usepackage{amsopn}

\DeclareMathOperator{\id}{Id}
\DeclareMathOperator*{\argmin}{argmin}

\DeclareMathOperator{\prox}{Prox} 
\DeclareMathOperator{\im}{im} 
\DeclareMathOperator{\dom}{dom} 
\newcommand{\inv}{^{-1}}             %inversion
\newcommand{\gobs}{g^\mathrm{obs}}

\newcommand{\nat}{\mathbb{N}_0}

\newcommand{\X}{\mathbb{X}}

\newcommand{\XA}{{\X_A}}
\newcommand{\XL}{{\X_L}}
\newcommand{\Y}{\mathbb{Y}}

\newcommand{\xh}{\hat{x}_\alpha}
\newcommand{\fh}{\hat{f}_\alpha}

\newcommand{\wav}{\mathcal{S}}
\newcommand{\bspace}[3]{b^{#1}_{{#2},{#3}}}
\newcommand{\Bspace}[3]{B^{#1}_{{#2},{#3}}(\Omega)}
\newcommand{\Bn}[4]{\| {#1} \|_{ B^{#2}_{{#3},{#4}}(\Omega)}}
\newcommand{\bn}[4]{\| {#1} \|_{{#2},{#3},{#4}}}
\newcommand{\lspace}[2]{\ell_{#1}^{#2}}
\newcommand{\lnorm}[3]{\| {#1} \|_{{#2},{#3}}}
\newcommand{\wspace}[3]{\ell_{{#1},{#2}}^{{#3},\infty}}
\newcommand{\wnorm}[4]{\| {#1} \|_{{#2},{#3},{#4}}}
\newcommand{\smax}{s_\textrm{max}}
\ifpdf
\hypersetup{
  pdftitle={Variational Regularization Theory Based on Image Space Approximation Rates},
  pdfauthor={P. Miller}
}
\fi
\newcommand{\newww}[1]{#1}
\newcommand{\neww}[1]{#1}
\begin{document}
\maketitle
% REQUIRED
\begin{abstract}
We present a new approach to convergence rate results for variational regularization. Avoiding Bregman distances and using image space approximation rates as source conditions we prove a nearly minimax theorem showing that the modulus of continuity is an upper bound on the reconstruction error up to a constant. Applied to Besov space regularization we obtain convergence rate results for $0,2,q$- and $0,p,p$-penalties without restrictions on $p,q\in (1,\infty).$ Finally we prove equivalence of H\"older-type variational source conditions, bounds on the defect of the Tikhonov functional, and image space approximation rates. 
\end{abstract}
% REQUIRED
% REQUIRED
\section{Introduction}
The subject of this paper are ill-posed equations $Ax=g$ \neww{ with $A$ a bounded linear operator mapping from a Banach space $\X$ to a Hilbert space $\Y$}. \neww{We analyze approximations of an unknown $x\in \X$ given noisy, indirect observations $g^\delta$ satisfying $\|g^\delta - Ax \|_\Y\leq \delta$ with a fixed noise level $\delta>0$.} \\ 
In this context ill-posedness means that the unknown $x$ does not depend continuously on the observations $g^\delta$. As a naive application of the inverse of $A$ may therefore amplify the noise indefinitely regularization is needed to compute stable approximations of the unknown. Here, 
we study variational regularization with a convex penalty $\mathcal{R}$ defined \neww{on $\X$}. 
More precisely, we consider the Tikhonov functional given by 
  \[T_\alpha(x,g):= \frac{1}{2 \alpha} \|g- Ax  \|_{\Y}^2  + \mathcal{R}(x) \quad\text{for } \alpha >0, x\in \dom(\mathcal{R}) \text{ and } g\in \Y \]
and denote its set of minimizers by
\begin{align*}
R_\alpha (g) &  := \argmin\nolimits_{x \in \dom(\mathcal{R})} \, T_\alpha(x,g) \subseteq \dom(\mathcal{R}).
\end{align*}
%A central aim of regularization theory are upper bounds on the distance \neww{$\|x-\xh\|_\X$}  between $x$ and estimators $\xh\in R_\alpha(g^\delta)$.\\
A central aim of regularization theory are upper bounds on the distance $L(x,\xh)$  between $x$ and estimators $\xh\in R_\alpha(g^\delta)$ with respect to some loss function $L$.
For ill-posed problems the convergence of $\xh$ to $x$ for $\delta\rightarrow 0$ can be arbitrarily slow in general. Therefore, upper bounds on the error require regularity conditions on the \neww{true solution} $x$, which are referred to as \emph{source conditions} in regularization theory. The name comes from the first such conditions in a Hilbert space setting, $x=(A^\ast A)^{\nu/2}\omega, \nu>0$, where $\omega$ is referred to as source generating $x$. This condition implies the convergence rate $\|x-\xh\|_\X= \mathcal{O}(\delta^\frac{\nu}{\nu+1})$ in the Hilbert space norm that defines the penalty. In \cite{Math2003} convergence rates in Hilbert scales are proven \neww{under source conditions of the form $x=\varphi(A^\ast A)\omega$ for more general functions $\varphi$}. Nevertheless, we restrict our attention to H\"older-type convergence rates in this paper. A generalization of the above source condition for $\nu=1$ to convex or Banach space penalties is given by \emph{source-wise representations} 
\begin{align}\label{eq:classical_sc}
A^\ast \omega \in \partial\mathcal{R}(x) \quad \text{for some } \omega\in \Y
\end{align} 
leading to the convergence rate $\mathcal{O}(\delta)$ in the Bregman divergence of $\mathcal{R}$ (see \cite{BO:04}). Slower rates of convergence in Banach space settings can be shown under \emph{variational source conditions} \cite{Scherzer_etal:09,Schuster2012} or under \emph{approximate source conditions} \cite{hein2008,hein2009}. We refer to \cite{flemming:12b} for a comparison of the latter two concepts. Recently in \cite{Grasmair2020} convergence rates are shown under the condition $(A^\ast A)^\nu \omega \in \partial\mathcal{R}(x)$ for convex penalties defined on Hilbert spaces. 
In \cite{Hofmann2019} upper bounds on $T_\alpha(x,Ax)-T_\alpha(x_\alpha,Ax)$ (\emph{defect of the Tikhonov functional}) in terms of $\alpha$ are used as a source condition.  \\
In this work we consider \emph{H\"older-type image space approximation rates}, i.e. bounds of the form   
\begin{align}\label{eq:image_space_approx_bound}
\|Ax-Ax_\alpha\|_\Y\leq c \alpha^\nu \quad\text{for all}\quad \alpha>0,\,  x_\alpha\in R_\alpha(Ax)
\end{align}
for some $\nu \in [0,\infty)$ and $c\geq 0$. These will play the role of source conditions.  
In many situations these kind of bounds can be proven under source conditions. (see e.g. \cite[Thm.~2.3]{HW:17}, \cite[Prop.~6]{Hofmann2019}). \\
%the same conditions on $x$ as convergence rates, i.e. bounds on the distance between $x$ and $\xh\in R_\alpha(g^\delta)$ 
We first prove a bound on $L(x,\xh)$ uniformly on the set of all $x$ satisfying \eqref{eq:image_space_approx_bound} in terms of the \emph{modulus of continuity}. 
One main advantage of our analysis is the flexibility in the choice of the loss function $L$. Then 
for penalties given by Banach space norm powers we work out a characterization of condition \eqref{eq:image_space_approx_bound} in terms of real interpolation spaces. This leads to convergence rate results with regularity conditions given by real interpolation spaces. As examples we consider weighted $\ell^p$-regularization, Besov space $0,p,p$- and $0,2,q$-regularization. Our approach seems to allow for the first time to obtain minimax optimal rates for all $p$ resp. $q$ in $(1,\infty).$
Finally we compare condition \eqref{eq:image_space_approx_bound} to source conditions used in the literature. We prove equivalence of \eqref{eq:image_space_approx_bound}, H\"older-type variational source conditions (used e.g. in \cite{Kindermann2016, F:18}) and H\"older-type bounds on the defect of the Tikhonov functional. \neww{In particular, this equivalence yields a \newww{characterization} of \eqref{eq:image_space_approx_bound} that does not directly depend on the minimizers $x_\alpha$.} \\ 
The structure of the paper is as follows: In Section 2 we present our main results. Section 3,4 and 5 are devoted for the proofs of the main results and establish some new techniques which may be of some independent interest in variational regularization theory. We finish with an outlook where we also discuss limitations of the present work.  
\section{Main results}
\neww{
To give an overview over the main results of this paper, we present and discuss the theorems in their precise mathematical form and refer to the proofs given in the sections 3, 4 and 5.} 
\subsection{Minimax convergence rates}\label{sec:main_minimax}
\begin{assumption}\label{ass:topological}
Let $\tau$ be a topology such that \neww{$(\X,\tau)$} is a locally convex Hausdorff space and $\mathcal{R}\colon \neww{\X} \rightarrow (-\infty,\infty]$ a proper, convex function. We assume that the sublevel set $\{ x\in \X \colon \mathcal{R}(x) \leq \lambda \}$ is $\tau$-compact for all $\lambda\in \mathbb{R}$. Note that this implies that $\mathcal{R}$ is lower semi-continuous on $(\X,\tau)$. \\
Let $\Y$ be a Hilbert space and $A\colon \neww{\X} \rightarrow \Y$ a linear, $\tau$-to-weak continuous operator.
\end{assumption}
\Cref{ass:topological} implies $\tau$-compactness of the sublevel sets of the Tikhonov functional. Using the finite intersection property of these sets one can show that $R_\alpha (g)$ is nonempty for all $g\in \Y$. Furthermore, for every $g\in \im(A)$ there exist a (possibly not unique) $\mathcal{R}$-minimal \neww{$x\in \X$} with $Ax=g$\neww{, i.e.\ $\mathcal{R}(x)\leq \mathcal{R}(z)$ for all $z\in \X$ with $Az=g$.}
Let $\nu\in [0,\infty)$ and $\varrho>0$. We define 
\begin{align}\label{eq:def_K}
\varrho_\nu \colon  \X  \rightarrow & [0,\infty]  \quad \text{by}\quad \varrho_\nu(x)= \sup \left\{ \alpha^{-\nu} \|Ax-Ax_\alpha\|_\Y \,\colon \alpha>0 ,\, x_\alpha\in R_\alpha(Ax)\right\},\\ \nonumber
 K_\nu ^\varrho & =\left\{ x\in \X \colon \varrho_\nu (x)\leq \varrho \right\} \quad\text{and}\quad  
 K_\nu := \left\{ x \in \X\, \colon \varrho_\nu(x) <\infty \right\}.
\end{align}
Note that $x\in K_\nu$ if and only if a bound \eqref{eq:image_space_approx_bound} holds true and  $\varrho_\nu(x)$ is the smallest possible constant $c>0$.\\
\neww{Let $L\colon \X\times \X \rightarrow [0,\infty]$  satisfy the triangle inequality. We use $L$ to measure the reconstruction error.}\\
The first and central result is a uniform bound in $K_\nu^\varrho$ on $L(x,\xh)$ with $\xh\in R_\alpha(g^\delta)$ in terms of the  \emph{modulus of continuity}. 
Recall that the latter is given by 
\begin{align} \label{eq:modulus}
\Omega(\delta,K) := \sup \left\{ L(x_1, x_2)  \, \colon x_1,x_2 \in K \text{ with  } \|Ax_1-A x_2\|_\Y\leq \delta\right\} 
\end{align} 
for a subset $K \subset \X$. \\
We consider two parameter choice rules for the regularization parameter $\alpha.$ An apriori rule requiring prior knowledge of the parameter $\nu$ in \eqref{eq:image_space_approx_bound} characterizing the regularity of the unknown $x$, and the discrepancy principle as most well-known a-posteriori rule. 
\begin{thm}\label{thm:minimax}
Let $\nu\in (0,1]$ and $\varrho, \alpha>0$. 
Suppose $x\in K_\nu^\varrho$. Let $\alpha>0$ and $\xh\in R_\alpha(g^\delta)$. 
\begin{remunerate}
\item (apriori rule) Let $c_r \geq c_l>0$. If $c_l\varrho^{-\frac{1}{\nu}} \delta^\frac{1}{\nu} \leq \alpha \leq c_r  \varrho^{-\frac{1}{\nu}}\delta^\frac{1}{\nu}$, then 
\[L(x,\xh)  \leq \Omega(c_1\delta, K_\nu^{c_2 \varrho}) \] with 
$c_1:= 1+c_r^\nu$ and $c_2:=2+c_l^{-\nu}$. 
\item (discrepancy principle) Let $C_D>c_D>1$. If $c_D\delta \leq \|g^\delta- A\xh \|_\Y\leq C_D\delta$, then 
\[L(x,\xh)  \leq \Omega(d_1 \delta, K_\nu^{d_2 \varrho}) \] with 
$d_1:= 1+C_D$ and $d_2:=2+(c_D-1)\inv$. 
\end{remunerate}
\end{thm}
\neww{The \hyperref[pr:thm_minimax_bound]{proof of Theorem 1}  can be found in \Cref{sec:minimaxity_on}.}
Under mild assumptions \Cref{thm:minimax} gives rise to an almost minimax result in the following manner. Recall that the worst case error of a reconstruction map $R\colon \Y\rightarrow \X$ on a set $K\subset \X$ is given by
\[ \Delta_R(\delta,K):= \sup \left\{ L\left( x ,R(g^\delta)\right))  \colon x\in K, g^\delta\in \Y \text{ with } \|g^\delta - Ax\|_\Y \leq \delta \right\}.  \]
and satisfies the lower bound 
 \begin{align}\label{eq:lower_bound_modulus}
\Delta_R(\delta,K) \geq \frac{1}{2} \Omega(2\delta,K)
\end{align} 
(see {\cite[Rem.~3.12]{EHN:96}}, {\cite[Lemma 3.11]{WSH:20}} or \cite[4.3.1. Prop.~1]{DJKP:95}).
Let $\overline{R}_\alpha\colon \Y \rightarrow \X$ satisfy $\overline{R}_\alpha(g^\delta)\in R_\alpha(g^\delta)$ for all $g^\delta\in \Y$ with either $\alpha=\alpha(\delta)$ satisfying the  apriori parameter choice given in \Cref{thm:minimax}.1. or $\alpha=\alpha(\delta,g^\delta)$ satisfying the discrepancy principle in  \Cref{thm:minimax}.2. 
%For both - a suitable apriori parameter choice of $\alpha$ and the discrepancy principle - the bound we prove under the assumption $x\in K_\nu^\varrho$ is of the form  
%\[ L(x,\xh) \leq \Omega(c_1 \delta, K^{c_2 \varrho}_\nu) \] 
%with constants $c_1, c_2$ independent of $x$. 
In the case $\Omega(\delta,K_\nu^\varrho)\sim \varrho^{e}\delta^{f} $ for some exponents $e,f>0$ this yields a minimax result 
\[ \Delta_{\overline{R}_\alpha}(\delta,K_\nu^\varrho) \leq  C \inf\nolimits_{R} \Delta_R(\delta,K_\nu^\varrho).\] 
This shows that up to a constant $C$ no method can achieve a better approximation uniformly on $K_\nu^\varrho.$ \\
Moreover, we would like to highlight the flexibility in the choice of \neww{ the loss function $L$}. Many recent works in Banach space or convex regularization theory are restricted to error bounds in the Bregman divergence (see e.g. \cite{Kindermann2016}, \cite{Hofmann2019}, \cite{F:18}, \cite{Werner2012}). In some situations the meaning of the Bregman divergence is unclear and lower bounds on the Bregman distance are required to obtain more tangible statements. In \cite{WSH:20} these lower bounds cause a restriction on the parameters $s,p,q$ of the Besov scale. By applying \Cref{thm:minimax} to Besov space regularization we can overcome these restrictions.
\subsection{Convergence rate theory for Banach space regularization}\label{sec:main_results_banach}
%\begin{align}\label{eq:classical_sc}
%\im(A^\ast)\cap\partial\mathcal{R}(x)\neq \emptyset.
\neww{Here we consider $\mathcal{R}\colon \X \rightarrow [0,\infty)$ given by 
$\mathcal{R}(x)=\frac{1}{u}\|x\|_\X^u$ for fixed $u\in [1,\infty)$.
%\end{align} 
We assume $\XA$ to be a Banach space with a continuous, dense embedding $\X\subset \XA$ such that $A$ extends to a norm isomorphism $
A\colon \XA\rightarrow \Y$, i.e. there exists a constant $M\geq 1$ such that
\begin{align} \label{eq:XAequiA}
\frac{1}{M} \|x\|_\XA \leq \|Ax\|_\Y \leq M \|x\|_\XA \quad\text{for all} \quad x\in \XA.
\end{align} 
Note that injectivity is necessary for \eqref{eq:XAequiA}. On the other hand injectivity of $A\colon \X\rightarrow \Y$ suffices for the existence of a space $\XA$ such that \eqref{eq:XAequiA} holds with $M=1$. (Take the Banach completion of $\X$ in the norm $x\mapsto \|Ax\|_\Y$).}\\
\neww{For example, in Besov space settings we will assume $\XA$ a space with negative smoothness index, and we consider spaces $\X$ with smoothness index $0$.\\
Moreover we need the following assumption on \newww{$K_1$ and $\varrho_1$ defined in \eqref{eq:def_K}}.
Recall that a quasi-norm satisfies the properties of norm except that the triangle inequality is replaced by $\|x+y\|\leq c \left(\|x\|+\|y\| \right)$ for a constant $c>0$. A complete and quasi-normed vector space is called a quasi-Banach space.}
\begin{assumption} \label{ass:K_1_quasi_banach}
Let $u\in (0,\infty)$.
\newww{Suppose $K_1$ is a vector space and that there is a quasi-norm $\|\cdot \|_\textrm{lin}$ on $K_1$ such that $(K_1,\|\cdot \|_\textrm{lin})$ is a quasi-Banach space. Moreover assume}
\[\frac{1}{M} \varrho_1 (x) \leq \|x \|_\textrm{lin} ^{u-1} \leq M \varrho_1(x) \quad\text{for all } x\in K_1.\] 
\end{assumption} 
This assumption is motivated by the computation of $K_1$ for the examples below.\\ 
%Under \eqref{eq:XAequiA} and \Cref{ass:K_1_quasi_banach} we show a converse result for image space approximation rates. More precisely we prove equivalence of the scale $K_\nu$ and the scale of real interpolation spaces $\left(\XA,K_1\right)_{\theta,\infty}$ (\Cref{prop:K_nu_under_quasi_assumption}). Appling \Cref{thm:minimax} leads to the following result:
\neww{
Recall that for a quasi-Banach space $\X_S$ with a continuous embedding $\X_S\subset \XA$ and $\theta\in (0,1)$ the real interpolation space $(\XA,\X_S)_{\theta,\infty}$ consists of all $x\in \XA$ such that 
\[ \|x\|_{(\XA,\X_S)_{\theta,\infty}} := 
\sup\nolimits_{t>0} t^{-\theta} K(x,t) <\infty . \]  
Here the $K$-functional is given by 
\[  K(x,t):= \inf\nolimits_{z\in \X_S} \left(\|x-z\|_{\X_A} + t \|z\|_{\X_S} \right) .\] }
For the definition of the real interpolation spaces $(\XA,\X_S)_{\theta,q}$ for $q\in (0,\infty)$ we refer to \cite{Bergh1976}. 
\begin{thm}[error bounds] \label{thm:error_bounds}
Suppose \eqref{eq:XAequiA} and \Cref{ass:K_1_quasi_banach} hold true. \neww{If $\X$ is not reflexive, suppose \Cref{ass:topological}} holds true. \\ 
\neww{Let $\XL$ be a Banach space with a continuous embedding $\XL\subset \XA.$}
Let $0<\xi< \theta <1$ and $\delta,\varrho,\alpha>0$ and $c_r\geq c_l >0$, $C_D>c_D>1$. Suppose there there is a continuous embedding $\left(\XA, K_1 \right)_{\xi,1}\subset \XL$. Assume \[ x\in \left(\XA, K_1 \right)_{\theta,\infty} \quad \text{with}\quad  \|x\|_{\left(\XA, K_1 \right)_{\theta,\infty}}\leq \varrho. \] Let $\xh\in R_\alpha(g^\delta)$. There exists a constant $C>0$ independent of $x,\delta$ and $\varrho$ such that whenever $\alpha$ satisfies either 
\[ c_l \varrho^{-\frac{u-1}{\theta}} \delta^\frac{(1-\theta)(u-1)+\theta}{\theta} \leq \alpha \leq c_r \varrho^{-\frac{u-1}{\theta}} \delta^\frac{(1-\theta)(u-1)+\theta}{\theta}  \quad\text{or} \quad 
 c_D\delta\leq \|g^\delta- A\xh \|_\Y \leq  C_D \delta \] 
the bound 
\[ \|x-\xh \|_{\XL} \leq C \varrho^\frac{\xi}{\theta} \delta^{1-\frac{\xi}{\theta}} \] 
holds true.
\end{thm}
\neww{We refer to \Cref{sec:error_Banach} for the \hyperref[pr:thm_error_Banach]{proof of Theorem 2}.}
\begin{remark}\label{rem:limiting_case}
The statement of the theorem remains valid in the limiting case $\theta=1$ where the source condition in terms of $(\XA,K_1)_{\theta,\infty}$ has to be replaced by simply $x\in K_1$ with $\|x \|_\textrm{lin}\leq \varrho$. Here the apriori rule is $\alpha \sim \varrho^{-(u-1)} \delta$. 
\end{remark} 
We illustrate the impact of this result by applying it to three more concrete Banach space regularization setups. 
%In all three examples we consider a further Banach space \neww{$\XR\subset \XA$, $u\in [1,\infty)$ and $\mathcal{R}\colon \XR\rightarrow [0,\infty]$ given by $\mathcal{R}(x)=\frac{1}{u}\|x\|_\XR^u$ if $x\in \XR\cap \X$ and $\mathcal{R}(x)=\infty$ if $x\in \X\setminus \XR$. }
\subsubsection*{Example 1: Weighted \texorpdfstring{$p$}{}-Norm Penalization} \label{ex:ex_1}
Let $\Lambda$ be a countable index set, $p\in (0,\infty)$ and $\overline{\omega}=(\overline{\omega}_j)_{j\in \Lambda}$ a sequence of positive reals. We consider weighted sequence spaces $\lspace {\overline{\omega}} p$ defined by 
\[ \lspace {\overline{\omega}}  p = \left\{ x\in \mathbb{R}^\Lambda \colon \lnorm x {\overline{\omega}}  p < \infty \right\} \quad\text{with } \quad \lnorm x {\overline{\omega}}  p ^p = \sum\nolimits_{j\in \Lambda} {\overline{\omega}}_j^p | x_j |^p . \] 
We assume that the forward operator maps a weighted $\ell^2$-space isomorphically to the image space $\Y$. More precisely, we suppose that \eqref{eq:XAequiA} holds true with $\XA=\lspace  {\overline{a}} 2$  for $\overline{a}=(\overline{a}_j)_{j\in \Lambda}$ a sequence of positive real numbers.\\
Moreover let $p\in (1,2)$ and $\overline{r}=(\overline{r}_j)_{j\in \Lambda}$ a sequence of weights such that $\overline{a} \overline{r}\inv$ is bounded. We consider $\X=\lspace {\overline{r}}p \subset\lspace  {\overline{a}} 2$ (see \cite[Prop.~A.1.]{MH:20}) with $\mathcal{R}(x)=\frac{1}{p} \lnorm x {\overline{r}} p ^p.$ \\
Furthermore we introduce weighted weak $\ell^p$-spaces.
For $\mu=  (\mu_j)_{j\in \Lambda}$ and $\nu= (\nu_j)_{j\in \Lambda}$ sequences of positive reals and $t\in (0,\infty)$ those are defined by the following quasi-norms
\[ \wspace \mu\nu t = \{ x\in \mathbb{R}^\Lambda \colon \wnorm x \mu\nu t <\infty \} \quad\text{ with } \wnorm x \mu\nu t ^t = \sup_{\tau>0} \left( \tau^t \sum\nolimits_{j\in \Lambda} \nu_j \mathds{1}_{\{ \mu_j |x_j| >\tau \}} \right) .\] 
We apply \Cref{thm:error_bounds} and obtain the following result. 
\begin{corollary}[error bounds for weighted $p$-norm penalties] \label{thm:rates_weighted}
Let \newww{$p\in (1,2)$,} $t\in (2p-2,p)$ and $\delta, \varrho,  \alpha>0$ and $c_r\geq c_l>0$, $C_D>c_D>1$ and $\mu:=(\overline{a}^2 \overline{r}^{-p})^\frac{1}{2-p} ,\quad\nu:=(\overline{a}\inv \overline{r})^\frac{2p}{2-p}$. Assume  $x\in \wspace \mu\nu t$ with $\wnorm x \mu\nu t\leq \varrho$ and $\xh\in R_\alpha(g^\delta).$
There is a constant $C>0$ independent of $x$, $\delta$ and $\varrho$ such that whenever $\alpha$ satisfies either  
\[ c_l \varrho^{-\frac{t(2-p)}{2-t}} \delta^\frac{2(2-p)}{2-t} \leq \alpha \leq c_r \varrho^{-\frac{t(2-p)}{2-t}} \delta^\frac{2(2-p)}{2-t}
\quad\text{ or }\quad
 c_D\delta\leq \|g^\delta- A\xh \|_\Y \leq  C_D \delta \] 
the bound 
\[ \lnorm {x-\xh}{\overline{r}} p \leq C \varrho^\frac{t(2-p)}{p(2-t)}  \delta^\frac{2(p-t)}{p(2-t)}.\] 
holds true.
\end{corollary} 
\neww{
The \hyperref[pr:weighed]{proof of Corollary 2.4} can by found in \Cref{sec:error_Banach}.}
\begin{remark}\label{rem:limiting_wei}
In the limiting case $t=2p-2$ the statement remains valid if one replaces $\wspace \mu\nu t$ by $K_1= \lspace {\overline{s}} {2p-2}$ with $\overline{s}=\overline {a}^{-\frac{1}{p-1}}\overline{r}^\frac{p}{p-1}$. Here we obtain the rate 
$ \lnorm {x-\xh}{\overline{r}} p \leq C \varrho^\frac{p-1}{p}  \delta^\frac{1}{p}.$
\end{remark}
In \cite{grasmair:10} the rate $\mathcal{O}(\delta^\frac{1}{p})$ is already proven under a condition similar to \eqref{eq:classical_sc}. Here we obtain intermediate convergences rates between $\mathcal{O}(\delta^0)$ and $\mathcal{O}(\delta^\frac{1}{p})$. This has the advantage that we obtain statements on the speed of convergences on larger sets.
\begin{remark} 
\Cref{thm:rates_weighted} remains valid word by word in the case $p=1$ (see \cite[Thm.~4.4]{MH:20}).
\end{remark} 
\subsubsection*{Example 2: Besov \texorpdfstring{$0,p,p$}{}-Penalties}\label{ex:0pp}
We introduce a scale of sequence spaces that allows to characterize Besov function spaces by decay properties of coefficients in wavelet expansions (see \cite{triebel:08}). \\
Let $(\Lambda_j)_{j\in\nat}$ be a family of sets such that 
$2^{jd}\leq |\Lambda_j|\leq C_\Lambda 2^{jd}$ for some constant $C_\Lambda\geq 1$ and all $j\in\nat$. We consider the index set $\Lambda:= \{(j,k) \colon j\in\mathbb{N}_0, k\in\Lambda_j\}$.\\
For $p,q \in (0,\infty)$ and $s\in \mathbb{R}$ we set
$
\bspace spq=\left\{x\in\mathbb{R}^\Lambda\colon \bn xspq<\infty\right\} $ with \[ 
\bn xspq^q:= \sum\nolimits_{j\in \nat} 2^{jq(s+\frac{d}{2}-\frac{d}{p})} \left( \sum\nolimits_{k\in \Lambda_j} |x_{j,k}|^p \right)^{q/p}.  
\]
with the usual replacements for $p=\infty$ or $q = \infty$. \\
%Note that for $p=q <\infty$ we obtain 
%\begin{align}\label{eq:besov_weighted_identification}
% \bspace spp =\lspace {\overline{\omega}_{s,p}}p \quad\text{with equal norm for}\quad (\overline{\omega}_{s,p})_{(j,k)}= 2^{j(s+\frac{d}{2}-\frac{d}{p})} 
%\end{align}
Let $a>0$ and assume that the forward operator $A\colon \bspace {-a}22\rightarrow \Y$ satisfies 
\eqref{eq:XAequiA} with $\XA= \bspace {-a}22$.
Let $p\in (1,\infty)$ (for $p=1$ we refer to \cite{MH:20} again) with $\frac{d}{p}-\frac{d}{2}\leq a$. Then we have a continuous embedding $\bspace 0pp\subset \bspace {-a}22$ (see \cite[3.3.1.(6),(7), 3.2.4.(1)]{Triebel2010}). \\  We use $\X=\bspace 0pp$ with  \[ 
\mathcal{R}(x)=\frac{1}{p} \bn x 0pp ^p = \frac{1}{p} \sum\nolimits_{(j,k)\in \Lambda} 2^{jp \left(\frac{d}{2}-\frac{d}{p}\right)}  |x_{j,k}|^p \quad\text{for }x\in \bspace 0pp. \]
\neww{Note that we have 
\begin{align}\label{eq:besov_weighted_identification}
 \bspace spp =\lspace {\overline{\omega}_{s,p}}p \quad\text{with equal norm for}\quad (\overline{\omega}_{s,p})_{(j,k)}= 2^{j(s+\frac{d}{2}-\frac{d}{p})}. 
\end{align}
Hence for $p<2$, this example is a special case of \hyperref[ex:ex_1]{Example 1}.}\\
Let $\tilde{s}=\frac{a}{p-1}$ and $\tilde{t}= 2p-2$. For $0<s<\tilde{s}$ we set 

\begin{align}\label{eq:k_s_interpol}
k_s := \left( \bspace {-a}22,  \bspace {\tilde{s}}{\tilde{t}}{\tilde{t}}\right)_{\theta,\infty} \quad\text{with}\quad \theta=\frac{p-1}{p}\frac{s+a}{a}.
\end{align}
Here the application of \Cref{thm:error_bounds} yields the following error bound. 
\begin{corollary}[error bounds for $0,p,p$-penalties] \label{thm:rates_besov_pp}
Let $0 < s< \tilde{s}$ and $\delta, \varrho,  \alpha>0$, $c_r\geq c_l>0$, $C_D>c_D>1$. Assume $x\in k_s$ with $\| x\|_{k_s}\leq \varrho$ and $\xh\in R_\alpha(g^\delta).$
There is a constant $C>0$ independent of $x$, $\delta$ and $\varrho$ such that whenever $\alpha$ satisfies either  
\[ c_l \varrho^{-\frac{pa}{s+a}} \delta^\frac{(2-p)s+2a}{s+a} \leq \alpha \leq c_r \varrho^{-\frac{pa}{s+a}} \delta^\frac{(2-p)s+2a}{s+a}
\quad\text{ or }\quad
 c_D\delta\leq \|g^\delta- A\xh \|_\Y \leq  C_D \delta \] 
the bound 
\[ \bn {x-\xh}0 p p \leq C \varrho^\frac{a}{s+a} \delta^\frac{s}{s+a}.\] 
holds true.
\end{corollary}
\neww{
The \hyperref[pr:pp]{proof of Corollary 2.7} can by found in \Cref{sec:error_Banach}.}
\begin{remark} \label{rem:limiting_pp}
In the limiting case $s=\tilde{s}$ the result remains valid if one replaces $k_s$ by $K_1= \bspace {\tilde{s}} {\tilde{t}} {\tilde{t}}$ and we obtain the bound
$\bn {x-\xh}0 p p \leq C  \varrho^\frac{p-1}{p}  \delta^\frac{1}{p}$.
\end{remark} 
For $p=2$ we have $k_s=\bspace s 2 \infty$ (see \cite[3.3.6.(9)]{Triebel2010}). The following proposition provides a nesting of $k_s$ for $p\neq 2$ by Besov sequence spaces.
\begin{proposition} \label{prop:nesting_pp}
Let $0<s<\tilde{s}$ and $t=\frac{2pa}{(2-p)s+2a}$. 
\begin{remunerate}
\item For $p<2$ we have continuous embeddings
\[ \bspace stt \subset k_s \subset \bspace s {t-\varepsilon} \infty \quad\text{for all}\quad 0< \varepsilon <t. \] 
\item For $p>2$ we have continuous embeddings $ \bspace stt \subset k_s \subset \bspace s {2} \infty.$ 
\end{remunerate}
\end{proposition} 
\neww{
We refer to \Cref{sec:equivalence} for a  \hyperref[pr:nesting]{proof of Proposition 2.9}.}
For $p<2$ the same argument as in \cite[Ex.6.7.]{MH:20} shows that describing the regularity of functions with jumps or kinks via their wavelet expansion in terms of $k_s$ allows for a higher value of $s$ then using $\Bspace ps\infty$ as in \cite{WSH:20}. Therefore we obtain a faster convergence rate for this class of functions. \\
For $p>2$ we measure the error in a stronger norm than the $\ell^2$-norm. On the other hand the set on which we obtain convergence rates is smaller than $\bspace s2\infty$.
\subsubsection*{Example 3: Besov \texorpdfstring{$0,2,q$}{}-Penalties}
Again we consider $a>0$ and $\XA=\bspace {-a}22$ with $A$ satisfying \eqref{eq:XAequiA}. Let $q\in (1,\infty)$. Then there is a continuous embedding ${\X:=\bspace 0 2 q \subset \bspace {-a} 2 2}$ (see  \cite[3.3.1.(7)]{Triebel2010}) and we choose \[ \mathcal{R}(x)= \frac{1}{q} \bn x 02q ^q = \frac{1}{q} \sum\nolimits_{j\in \mathbb{N}_0} \left(\sum\nolimits_{k\in\Lambda_j} |x_{j,k}|^2 \right)^{q/2}.\]  
For a convergence analysis in the case $q=1$ we refer to \cite{HM:19}. 
The application of \Cref{thm:error_bounds} provides: 
\begin{corollary}[error bounds for $0,2,q$-penalties] \label{thm:rates_besov_2q}
Let $0 < s< \frac{a}{q-1}$ and \linebreak ${\delta, \varrho,  \alpha>0}$, $c_r\geq c_l>0$, $C_D>c_D>1$. Assume $x\in \bspace s 2\infty$ with $\bn xs2\infty \leq \varrho$ and $\xh\in R_\alpha(g^\delta).$
There is a constant $C>0$ independent of $x$, $\delta$ and $\varrho$ such that whenever $\alpha$ satisfies either  
\[ c_l \varrho^{-\frac{qa}{s+a}} \delta^\frac{(2-q)s+2a}{s+a} \leq \alpha \leq c_r \varrho^{-\frac{qa}{s+a}} \delta^\frac{(2-q)s+2a}{s+a} 
\quad\text{ or }\quad
 c_D\delta\leq \|g^\delta- A\xh \|_\Y \leq  C_D \delta \] 
the bound 
\[ \bn {x-\xh}0 2 2 \leq C \varrho^\frac{a}{s+a} \delta^\frac{s}{s+a}.\] 
holds true.
\end{corollary} 
\neww{
The \hyperref[pr:2q]{proof of Corollary 2.10} can by found in \Cref{sec:error_Banach}.}
\begin{remark} \label{rem:limiting_2q}
In the limiting case $s=\frac{a}{q-1}$ the result remains valid if one replaces $\bspace s 2\infty$ by $K_1=\bspace {\tilde{s}}2{\tilde{q}}$ with $\tilde{q}=2q-2$. Here we obtain ${\bn {x-\xh}0 2 2 \leq  C \varrho^\frac{q-1}{q}  \delta^\frac{1}{q}}$.
\end{remark} 
In contrast to the analysis in \cite{HM:19} we measure the error in the $\ell^2$-norm independent of the value of $q$, i.e. the error norm is not dictated \neww{by} the penalty term.\\
The smaller $q$ the larger is the region $0<s<\frac{a}{q-1}$ of regularity parameters for which we guarantee upper bounds. Furthermore we see that changing the fine index $q$ while keeping $p=2$ does not change the set where convergence rates are \neww{guaranteed}, but it influences the parameter choice rule.
\neww{ 
\subsubsection*{Example 4: Radon Transform} \label{ex:02q}
To give a more concrete example we discuss the Radon transform which appears as forward operator in computed tomography (CT) and positron emission tomography (PET). This example also shows how our results apply to operators initially defined on function spaces.\\
Let $d\in \mathbb{N}$ with $d\geq 2$, $\Omega:=\{x\in\mathbb{R}^d \colon |x|\leq 1 \}$, $S^{d-1}:=\{ x\in \mathbb{R}^d \colon |x|=1\} $ and $\Y:=L^2(S^{d-1}\times [-1,1])$. Then the Radon transform ${R:L^2(\Omega)\to L^2(S^{d-1}\times \mathbb{R})}$ is given by
\[
(Rf)(\theta,t) := \int_{x\cdot \theta = t} f(x)\,dx,\qquad \theta\in S^{d-1},
t\in\mathbb{R}.
\] 
With $a=\frac{d-1}{2}$ it follows from \cite[Thm.~3.1]{hertle:83} 
that $R$ is a norm isomorphism from ${B^{-a}_{2,2}(\Omega)}$ 
to $\Y$. Here ${B^{-a}_{2,2}(\Omega)}$ denotes a Besov function space. We refer to the book \cite{GN:15} for an introduction to this scale of function spaces. \\ 
Furthermore, with $\Lambda$ and the scale of spaces $\bspace spq$ as introduced in \hyperref[ex:0pp]{Example 2} and $\smax>a$ we consider a $\smax$-regular wavelet system $(\psi_\lambda)_{\lambda\in\Lambda}$ on $\Omega$ such that the synthesis operator \[ \wav \colon\bspace spq \rightarrow \Bspace spq \quad\text{ given by } x\mapsto \sum_{\lambda\in \Lambda} x_\lambda \psi_\lambda  \] 
is well defined and a norm isomorphism for all $s\in \mathbb{R}$ and $p,q\in (0,\infty]$ satisfing $s \in (\sigma_p-s_\textrm{max}, s_\textrm{max})$ with $\sigma_p=\max\left\{d\left(\frac{1}{p}-1\right), 0 \right\}$ (see \cite{triebel:08}).
Now for $\mathcal{R}$ as in \hyperref[ex:0pp]{Example 2} or \hyperref[ex:02q]{Example 3} we consider 
\begin{align}\label{eq:Tky_wav}
S_\alpha(g) = \wav \hat{x}_\alpha \quad\text{with}\quad  \hat{x}_\alpha \in \argmin_{x\in \dom(\mathcal{R})} \left( \frac{1}{2\alpha } \|\gobs - R\wav x\|_{\Y}^2 + \mathcal{R}(x) \right)
\end{align}
and obtain the following convergence rate results.
\begin{corollary}[Convergence rates for wavelet regularization of the Radon transform] \label{cor:radon}
\begin{enumerate}
\item Let $p\in (1,\infty)$. With the notation of \hyperref[ex:0pp]{Example 2} suppose $0<s< \min\{\overline{s}, \smax\}$  and $\delta, \varrho,  \alpha>0$, $c_r\geq c_l>0$, $C_D>c_D>1$. Assume $f\in \Bspace stt$ with $\Bn fstt \leq \varrho$ and $\fh\in S_\alpha(g^\delta)$ with $\mathcal{R}$ as in \hyperref[ex:0pp]{Example 2}.
Then there is a constant $C>0$ independent of $f$, $\delta$ and $\varrho$ such that whenever $\alpha$ satisfies either  
\[ c_l \varrho^{-\frac{pa}{s+a}} \delta^\frac{(2-p)s+2a}{s+a} \leq \alpha \leq c_r \varrho^{-\frac{pa}{s+a}} \delta^\frac{(2-p)s+2a}{s+a}
\quad\text{or}\quad
 c_D\delta\leq \|g^\delta- R\fh \|_\Y \leq  C_D \delta \] 
the bound 
\[ \Bn {f-\fh}0pp  \leq C \varrho^\frac{a}{s+a} \delta^\frac{s}{s+a}.\] 
holds true.
If $p\leq 2$ we also obtain the bound
\[ \| f-\fh \|_{L^p(\Omega)} \leq C \varrho^\frac{a}{s+a} \delta^\frac{s}{s+a}.\] 
\item Let $q\in (0,\infty).$ Suppose $0<s< \min\left\{\frac{a}{q-1}, \smax\right\}$ and $\delta, \varrho,  \alpha>0$, $c_r\geq c_l>0$, $C_D>c_D>1$. Assume $f\in \Bspace s2\infty$ with $\Bn fs2\infty \leq \varrho$ and $\fh\in S_\alpha(g^\delta)$ with $\mathcal{R}$ as in \hyperref[ex:02q]{Example 3}.
Then there is a constant $C>0$ independent of $f$, $\delta$ and $\varrho$ such that whenever $\alpha$ satisfies either  
\[ c_l \varrho^{-\frac{qa}{s+a}} \delta^\frac{(2-q)s+2a}{s+a} \leq \alpha \leq c_r \varrho^{-\frac{qa}{s+a}} \delta^\frac{(2-q)s+2a}{s+a}
\quad\text{or}\quad
 c_D\delta\leq \|g^\delta- R\fh \|_\Y \leq  C_D \delta \] 
the bound 
\[ \| f-\fh \|_{L^2(\Omega)} \leq C \varrho^\frac{a}{s+a} \delta^\frac{s}{s+a}.\] 
holds true.
\end{enumerate}
\end{corollary} 
The \hyperref[pr:radon]{proof of Corollary 2.11} can be found in \Cref{sec:error_Banach}.
In \Cref{cor:radon}.1. it would also we sufficient to require $f\in \wav k_s$ instead of $f\in \Bspace stt$. Transfering the interpolation identity in \eqref{eq:k_s_interpol} to function spaces shows that $\wav k_s$ is independant of the choosen wavelet system (see \cite[Sec.~6.2]{MH:20} for a similar discussion). \\
In the same manner the presented theory can be applied to other linear, finitly smoothing forward operators as inverses of elliptic differential operators with smooth, periodic coefficients or specific periodic convolution operators (see \cite[Ex.~2.5]{HM:19} for more details).}
\subsection{Connections to source conditions}
Assuming only \Cref{ass:topological} we compare \eqref{eq:image_space_approx_bound} to source conditions used in the literature. 
For a concave and upper semi-continuous  function $\phi\colon [0,\infty)\rightarrow [0,\infty) $ we consider \emph{variational source conditions} of the form
\begin{align} \label{eq:vsc}
\mathcal{R}(x)-\mathcal{R}(z)\leq \phi(\|Ax-Az\|_\Y^2) \quad\text{ for all } z\in \X. 
\end{align}
In \cite{Kindermann2016} this condition is used to prove convergence rates with respect to the twisted Bregman distance of $\mathcal{R}$ and it is shown that the source condition \eqref{eq:classical_sc} implies \eqref{eq:vsc} with $\phi\sim \sqrt{\cdot}.$  In \cite{F:18} necessity of \eqref{eq:vsc} for convergence rates with respect to the twisted Bregman distance under a fixed parameter choice rule is proven. \\
Inspired by \cite{Hofmann2019} we also study the defect of the Tikhonov functional
\[  \sigma_x(\alpha):= T_\alpha(x,Ax)-T_\alpha(x_\alpha, Ax). \]
The following result shows that H\"older-type variational source conditions, H\"older-type bounds on the defect of the Tikhonov functional and H\"older type image space approximation rates are equivalent. 
\begin{thm}\label{thm:compare}
Let $\nu\in (\frac{1}{2},1]$. Assume $x\in \dom(\mathcal{R})$ is $\mathcal{R}$-minimal in $A\inv (\{ Ax \})$  and $x_\alpha \in R_\alpha(Ax)$ for $\alpha>0$ is any selection of a minimizers for exact data. The following statements are equivalent: 
\begin{romannum}
\item There exists a constant $c_1>0$ with $\|Ax-Ax_\alpha\|_\Y\leq c_1 \alpha^\nu$ for all $\alpha>0$. 
\item There exists a constant $c_2>0$ such that $\sigma_x(\alpha)\leq c_2 \alpha^{2\nu-1}.$
\item There exists a constant $c_3>0$ with \eqref{eq:vsc} holds true for $\phi(t)=c_3 t^\frac{2\nu-1}{2\nu} $.
\end{romannum} 
More precisely $(i)$ implies $(ii)$ with $c_2= \frac{c_1^2}{4\nu-2}$, $(ii)$ implies $(iii)$ with $c_3=2 c_2^\frac{1}{2\nu}$ and $(iii)$ implies $(i)$ with $c_1=c_3^\nu.$
\end{thm}
\neww{  
We provide a \hyperref[pr:2q]{proof of Theorem 3} in \Cref{sec:error_Banach}.}
%The connection between bounds on $\sigma_x$ and \eqref{eq:vsc} is sharp and works also for more general index functions (see \Cref{lem:vsc}). 
%A key step the differentiablity of the function $\sigma_x$ with 
%\[ \sigma_x^\prime(\alpha) = \frac{1}{2\alpha^2} \|Ax- Ax_\alpha\|_\Y^2. \]
The result allows \neww{the following representation of $K_\nu$ in terms of variational source conditions:
\begin{align}\label{eq:k_nu_representation_vsc}
 K_\nu = \left\{ x\in \X \colon \text{There exists }c>0 \text{ such that }  \eqref{eq:vsc} \text{ with }\phi(t)=c t^\frac{2\nu-1}{2\nu} \text{holds true.} \right\}
\end{align}
for all $\nu\in \nu\in (\frac{1}{2},1]$.
Note that since the map $ (\frac{1}{2},1]\rightarrow (0,\frac{1}{2}]$ given by $\nu \mapsto \frac{2\nu-1}{2\nu}$ is bijective this characterization grasps all H\"older type functions $\phi(t)=\mathcal{O}(t^\mu)$ for $\mu\in (0,\frac{1}{2}]$. Due to \cite[Prop. 3]{Hofmann2019} the largest meaningful exponent is $\mu=\frac{1}{2}$. Furthermore, \eqref{eq:vsc} implies  $x\in\dom(\mathcal{R})$ and this in turn yields $(i)$ with $\nu=\frac{1}{2}$. Therefore, we cannot expect a characterization of  $(i)$ for $\nu<\frac{1}{2}$ by variational source conditions.
Hence all meaningful H\"older type variatonal source conditions of the form \eqref{eq:vsc} are covered in \Cref{thm:compare} and \eqref{eq:k_nu_representation_vsc}. In other words it is not possible to extend \Cref{thm:compare} to a larger set of exponents.}\\
Together with \Cref{thm:minimax} we see that H\"older-type variational source conditions imply upper bounds on the reconstruction error for any loss function \neww{given by the modulus of continuity.} In contrast as far as the author knows all upper bounds in the literature derived from \eqref{eq:vsc} are restricted to the twisted Bregman distance. 
\section{Minimax convergence rates on  \texorpdfstring{$K_\nu$}{}}  \label{sec:mini_max}
The aim of this section is to prove \Cref{thm:minimax}. \neww{Here we only assume the topological assumptions given in \Cref{ass:topological}.}
%In this section we bound the approximation error of an unknown $x\in K_\nu$ by $\xh\in R_\alpha(g^\delta)$ given noisy data $g^\delta\in \Y$ satisfying a deterministic error bound $\|g^\delta - Ax \|_\Y\leq \delta$ \\
We will follow an idea presented in the seminal paper \cite{DJKP:95}: Any feasible procedure is nearly minimax (see \cite[4.3.1.]{DJKP:95}). 
In our context feasibility means
\begin{remunerate}
\item image space bounds:  $\|Ax - A\xh\|_\Y\leq c \delta$,
\item regularity of the minimizers:  $\varrho_\nu(\xh)\leq c \varrho_\nu (x)$ for some constant $c>0$. 
\end{remunerate}      
After proving feasibility we use the same argument as in \cite[4.3.1. Prop.~2]{DJKP:95} to obtain a nearly minimax result.
\subsection{Characterization of \texorpdfstring{$A \circ R_\alpha$}{} as proximity mapping} \label{sec:prox_mapping}
This subsection provides an important preliminary that we use in several places throughout the paper. We introduce a convex function $\mathcal{Q}$ on $\Y$ that can be seen as a push forward of $\mathcal{R}$ \neww{through} the linear operator $A$. We show that the proximity mapping of $\alpha \mathcal{Q}$ equals $A\circ R_\alpha$.
Recall that for a convex, proper and lower semi-continuous function $\mathcal{Q}\colon\Y \rightarrow (-\infty,\infty]$ and $g\in \Y$ there is a unique minimizer $\prox_\mathcal{Q}(g)$ of the function $y\mapsto \frac{1}{2} \| g -y  \|_\Y^2  + \mathcal{Q}(y)$. The \neww{single-valued} mapping \[ \prox_{\mathcal{Q}}\colon \Y\rightarrow \Y \quad\text{given by} \quad g\mapsto  \prox_{\mathcal{Q}}(g) \neww {:=\argmin_{y\in \Y}\left(\frac{1}{2} \| g -y  \|_\Y^2  + \mathcal{Q}(y) \right)}\]  is called proximity mapping of $Q$ (see \cite[11.4, Def.~12.23]{Bauschke}).
\begin{lemma}\label{prop:introduceQ}
We define 
\[ \mathcal{Q}\colon \Y\rightarrow (-\infty,\infty] \quad\text{by}\quad \mathcal{Q}(g):=\inf \{ \mathcal{R}(x) \, \colon \, \neww{x\in \X} \text{ with } Ax=g \}  \] with $\inf \emptyset=\infty$. Then $\mathcal{Q}$ is convex, proper and lower semi-continuous, and we have $\dom(\mathcal{Q})=A (\dom(\mathcal{R}))$. 
\end{lemma}
\begin{proof}
Let $\lambda \in \mathbb{R}$. First we prove that $L_\lambda:=\{g\in \neww{\Y}\,\colon\, \mathcal{Q}(g)\leq \lambda\}$ satisfies 
\begin{align*} 
 L_\lambda=A (\{\neww{x\in \X} \,\colon \, \mathcal{R}(x)\leq \lambda \}).
\end{align*}
To this end let $g\in L_\lambda$. There exists \neww{$x\in \X$} with $Ax=g$ and $\mathcal{R}(x)\leq \mathcal{R}(z)$ for all \neww{$z\in \X$} with $Az=g$. Then $\mathcal{R}(x)=\mathcal{Q}(g)\leq \lambda.$ On the other hand if \neww{$x\in \X$} with $\mathcal{R}(x)\leq \lambda$ then $\mathcal{Q}(Ax)\leq \mathcal{R}(x)\leq \lambda$.\\
Taking union over $\lambda\in \mathbb{R}$ yields $\dom(\mathcal{Q})= A(\dom(\mathcal{R})).$ Hence $\mathcal{Q}$ is proper as $\mathcal{R}$ is proper. 
The sublevel sets $L_\lambda$ are convex as the image of a convex set under a linear map and closed as the image of a $\tau$-compact set under a $\tau$-to-weak continuous map. Hence $\mathcal{Q}$ is convex and lower semi-continuous.
\end{proof}
\begin{remark} 
Note that in the case of an injective forward operator $A$, the map $\mathcal{Q}$ is given by $\mathcal{Q}(g)=\mathcal{R}(A\inv g)$ if $g\in \im(A)$ and $\mathcal{Q}(g)=\infty$ if $g\in \Y\setminus \im(A)$ where $A\inv \colon \im(A)\rightarrow \X$ denotes the inverse map of $A$. 
\end{remark} 

\begin{proposition}\label{lem:prox_of_Q}
Let $g\in \Y$ and $\alpha>0.$ Then 
\[ A \xh = \prox_{\alpha\mathcal{Q}}(g)  \quad\text{and}\quad \mathcal{R}(\xh)= \mathcal{Q}(\prox_{\alpha\mathcal{Q}}(g)) \quad\text{for all}\quad \xh\in R_\alpha(g). \] 
In particular $A\circ R_\alpha= \prox_{\alpha\mathcal{Q}}$ is single-valued. Hence $A\xh$ and $\mathcal{R}(\xh)$ do not depend on the particular choice of $\xh\in R_\alpha(g)$. 
\end{proposition} 
\begin{proof}
Let $v \in \dom(\mathcal{Q})$. By \Cref{prop:introduceQ} we have $v\in \im(A)$. There exists \neww{$z\in \X$} with $Az= v$ and $\mathcal{R}(z)\leq \mathcal{R}(y)$ for all \neww{$y\in \X$} with $Ay=v$. By definition of $\mathcal{Q}$ that is $\mathcal{R}(z)=\mathcal{Q}(v)$. The first identity follows from
\begin{align*}
\frac{1}{2\alpha} \|g-A\xh \|_\Y^2 +  \mathcal{Q}(A\xh)&  \leq \frac{1}{2\alpha} \|g-A\xh \|_\Y^2 +  \mathcal{R}(\xh) \\
&  \leq\frac{1}{2\alpha} \|g-Az \|_\Y^2 +  \mathcal{R}(z) \\
&  =\frac{1}{2\alpha}  \|g-v \|_\Y^2 + \mathcal{Q}(v).  
\end{align*}
Inserting $v=A\xh$ yields $\mathcal{R}(\xh)=\mathcal{Q}(A\xh)=\mathcal{Q}\left( \prox_{\alpha\mathcal{Q}}(g)\right)$.
\end{proof}
The statement in \Cref{lem:prox_of_Q} can be read as follows: the function $\mathcal{Q}$ on $\Y$ stores all relevant information on $\mathcal{R}$ and $A$ to recover the mapping $A\circ R_\alpha$ in one object. Note that the definition of  $K_\nu$ can be rephrased only in terms of $\mathcal{Q}$.   
\begin{remark}\label{rem:defect_tik_vs_R} 
Suppose $x\in \dom(\mathcal{R})$, $\alpha>0$ and $x_\alpha\in R_\alpha(Ax)$. In \cite{Hofmann2019} the authors study upper bounds on $\mathcal{R}(x)- \mathcal{R}(x_\alpha)$ (defect for penalty) and on $\sigma_x(\alpha)$ (defect for Tikhonov functional) in terms of $\alpha$. The first quantity bounds the second and it is bounded by the double of the second (see \cite[Prop.~2.4]{Hofmann2019}). In \cite[Rem.~2.5]{Hofmann2019} the authors rely on this nesting to argue that changing the selection of minimizers changes the defect for penalty at most by a factor of $2$.
\Cref{lem:prox_of_Q} actually shows that the defect for penalty is independent of the choice of $x_\alpha\in R_\alpha(Ax)$. 
\end{remark} 
Exploiting firm non-expansiveness (see \cite[Def.~4.1]{Bauschke}) of proximal operators we draw a further conclusion of \Cref{lem:prox_of_Q}.
\begin{corollary}[Firm non-expansiveness]  \label{prop:firmly}
Let $g,h \in \Y$, $\alpha>0$, $\xh\in R_\alpha(g)$ and  $\hat{z}_\alpha\in R_\alpha(h)$. Then   
\[ \|(g-A\xh) - (h-A\hat{z}_\alpha)\|_{\neww{\Y}}^2+ \|A\xh - A\hat{z}_\alpha\|_\Y^2 \leq \|g-h\|_\Y^2. \] 
\end{corollary} 
\begin{proof}
By \cite[Prop. ~12.27]{Bauschke} \neww{ the proximity operator $\prox_{\alpha\mathcal{Q}}$ satisfies 
\[ \|(g-\prox_{\alpha\mathcal{Q}}(g)) - (h-\prox_{\alpha\mathcal{Q}}(h))\|_\Y^2+ \|\prox_{\alpha\mathcal{Q}}(g) - \prox_{\alpha\mathcal{Q}}(h)\|_\Y^2 \leq \|g-h\|_\Y^2 \] 
for all $g,h\in \Y$. Inserting the first identity in \Cref{lem:prox_of_Q} yields the claim. }
\end{proof}

%The following corollary links the minimal value 
%$\inf_{x\in \X} T_\alpha(x,g)$ to the Moreau envelope of $\mathcal{Q}$. 
%
%\begin{corollary}\label{cor:Moreau_inf}
%Let $g\in \Y$ and $\alpha>0$. Then $\mathcal{Q}_\alpha(g)=\inf_{x\in \X} T_\alpha(x,g)$.
%\end{corollary} 
%\begin{proof}
%Let $\xh\in R_\alpha(g).$
%The chain of inequalites in the proof of \Cref{lem:prox_of_Q} yields 
%\[\inf_{x\in \X} T_\alpha(x,g) = \frac{1}{2\alpha} \|g-A\xh\|_\Y +\mathcal{R}(\xh) \leq \frac{1}{2\alpha} \|g-v\|^2 +\mathcal{Q}(v)\] 
%for all $v\in \dom(\mathcal{Q})$. 
%\end{proof}
%We finish this section with another corollary representing $\mathcal{Q}$ as the limit $\alpha\searrow 0$ of the minimal value of the Tikhonov functional. 
%\begin{corollary}\label{cor:minimal_limit}
%Let $g\in \Y$.  Then \[ \lim_{\alpha\searrow 0} \left(\inf_{x\in \X} T_\alpha(x,g) \right) = \mathcal{Q}(g).  \]  
%\end{corollary} 
%\begin{proof}
%Due to \ref{cor:Moreau_inf} and \cite[Prop. 12.32]{Bauschke} we have  
%\[  \inf_{x\in \X} T_\alpha(x,g)= \mathcal{Q}_\alpha(g) \rightarrow  \mathcal{Q}(g)  \quad\text{ for } \alpha \searrow 0.  \] 
%\end{proof}
\subsection{Properties of the sets \texorpdfstring{$K_\nu$}{} }
The following proposition captures properties of the sets $K_\nu$. In particular, we show that $K_\nu$ is nontrivial for $\nu\in (0,1]$.
\begin{lemma}\label{lem:K_nu_props}
We have
\begin{remunerate}
\item  \neww{$K_0=\X$}.
\item  $K_{\nu_2}\subset K_{\nu_1}$ for $0\leq \nu_1\leq \nu_2$. 
\item  $K_\nu=\argmin_{z\in \neww{\X}} \mathcal{R}(z)+\ker(A)$ for all $\nu>1$.
\item  $\dom(\mathcal{R})+\ker(A)\subseteq K_{1/2}$. 
\end{remunerate}
\end{lemma} 
\begin{proof}\hspace{-1pt}
\begin{remunerate} 
\item Let $x\in \X$. We set $D_x:= \inf\left\{ \|Ax-Ay\|_\Y \colon y\in \argmin_{z\in \neww{\X}} \mathcal{R}(z) \right\}.$ Let ${y\in \argmin_{z\in \neww{\X}} \mathcal{R}(z)}$, $\alpha>0$ and $x_\alpha\in R_\alpha(Ax)$. Then 
\[ \frac{1}{2\alpha} \|Ax - Ax_\alpha\|_\Y^2 +\mathcal{R}(x_\alpha)\leq \frac{1}{2\alpha} \|Ax - Ay \|_\Y^2 +\mathcal{R}(y). \]
As $\mathcal{R}(y)\leq \mathcal{R}(x_\alpha)$ this implies $\|Ax - Ax_\alpha\|_\Y \leq  \|Ax - Ay \|_\Y.$ Hence $\varrho_0(x)\leq D_x<\infty$. 
\item Suppose $x\in K_{\nu_2}$. Then 
\begin{align*}
\|Ax-Ax_\alpha\|_\Y & = \|Ax-Ax_\alpha\|_\Y^\frac{\nu_1}{\nu_2}  \|Ax-Ax_\alpha\|_\Y ^{1- \frac{\nu_1}{\nu_2}} \leq  \varrho_{\nu_2}(x)^\frac{\nu_1}{\nu_2} \varrho_0(x)^{1- \frac{\nu_1}{\nu_2}} \alpha^{\nu_1}. 
\end{align*}
implies $\varrho_{\nu_1}(x)\leq \varrho_{\nu_2}(x)^\frac{\nu_1}{\nu_2} \varrho_0(x)^{1- \frac{\nu_1}{\nu_2}} $.
\item Let $\nu>1$. Suppose $x\in K_\nu$. From \cite[Prop.~16.34]{Bauschke} and \Cref{lem:prox_of_Q} we obtain \[ \eta_\alpha:=\frac{1}{\alpha}\left( Ax-Ax_\alpha\right)=\frac{1}{\alpha}\left( Ax- \prox_{\alpha \mathcal{Q}}(Ax) \right)\in \partial\mathcal{Q}(Ax_\alpha).\]
Since $\eta_\alpha\rightarrow 0$ and $Ax_\alpha \rightarrow Ax$ for $\alpha\rightarrow 0$ in the norm topology of $\Y$ this implies $0\in \partial\mathcal{Q}(Ax)$. Hence $Ax\in \argmin_{g\in \Y}\mathcal{Q}(g)$. Let $\neww{y\in \X}$ be $\mathcal{R}$-minimal with $Ay=Ax.$ Then \[ \mathcal{R}(y)= \mathcal{Q}(Ax)\leq \mathcal{Q}(Az)\leq \mathcal{R}(z) \quad \text{for all }  \neww{z\in \X} .\]  Hence 
\[ x= y+ x-y  \in \argmin\nolimits_{z\in \neww{\X}} \mathcal{R}(z)+\ker(A).\] 
On the other hand assume $x=y + k \in\argmin_{z\in \neww{\X}} \mathcal{R}(z)+\ker(A)$. Then 
\[ \frac{1}{2\alpha} \|Ax-Ax_\alpha\|_\Y^2 +\mathcal{R}(x_\alpha) \leq T_\alpha(y, Ax) = \mathcal{R}(y) \leq \mathcal{R}(x_\alpha) \]
yields $\|Ax-Ax_\alpha\|_\Y=0.$ Hence $x\in K_\nu$.  
\item Let $x=y+k\in \dom(\mathcal{R})+\ker(A)$. From \[ \frac{1}{2\alpha} \|Ax-Ax_\alpha\|_\Y^2 \leq T_\alpha(x_\alpha,Ax)\leq T_\alpha(y,Ax)=\mathcal{R}(y)\] we obtain 
$\varrho_{1/2}(x)\leq \sqrt{2\mathcal{R}(y)}$. 
\end{remunerate}
\end{proof}
The set $K_\nu$ does not change for $\nu>1.$ As announced in \Cref{sec:main_results_banach} we will see that $K_1$ is the set of elements satisfying source condition \eqref{eq:classical_sc}.\\
Moreover note that the last inequality in the proof of \Cref{lem:K_nu_props}$.2.$ resembles an interpolation inequality. This gives a first hint to a connection to  interpolation theory in the case of Banach space regularization. 

\subsection{Image space bounds}
This subsection is devoted to error bounds in the image space $\Y$ in terms of the deterministic noise level and the image space approximation error for exact data. Let $\delta\geq 0$, \neww{$x\in \X$} and $g^\delta\in \Y$ with $\|g^\delta- Ax \|_\Y\leq \delta$. 
\begin{lemma}\label{prop:image_space_noisebounds} 
The following inequalities 
\begin{align}\label{eq:imbound_Ax}
\|Ax-A\hat{x}_\alpha \|_\Y & \leq \delta + \|Ax-Ax_\alpha\|_\Y,\\
\label{eq:imbound_g_delta}
\|g^\delta-A\hat{x}_\alpha\|_\Y &\leq \delta +\|Ax-Ax_\alpha\|_\Y 
\end{align}
hold true for all $\alpha>0$, $\hat{x}_\alpha\in R_\alpha(g^\delta)$, $x_\alpha\in R_\alpha(Ax)$. 
\end{lemma} 
%\begin{proposition}\label{prop:image_space_noisebounds} 
%Suppose $g\in \im(A)$, $\delta>0$ and $g^\delta\in \Y$ with $\|g-g^\delta\|_Y\leq \delta$. Then the following bounds 
%\begin{align*}
%\|g-A\hat{x}_\alpha \|_\Y & \leq \delta + \|g-Ax_\alpha\|_Y\quad\text{and}\\
%\|g^\delta-A\hat{x}_\alpha\|_\Y &\leq \delta +\|g-Ax_\alpha\|_Y 
%\end{align*}
%hold true for all $\alpha>0$, $\hat{x}_\alpha\in R_\alpha(g^\delta)$, $x_\alpha\in R_\alpha(Ax)$.
%\end{proposition} 

\begin{proof}
\Cref{prop:firmly} with $g=g^\delta$ and $h=Ax$ yields   
\begin{align*}
 \| (Ax-Ax_\alpha) - (g^\delta-A\hat{x}_\alpha)\|_\Y^2 + \| Ax_\alpha -A\hat{x}_\alpha\|_\Y^2 \leq \delta^2.
\end{align*}
We neglect the first summand on the left hand side and obtain
\[ \|Ax-A\hat{x}_\alpha \|_\Y  \leq \|Ax-Ax_\alpha \|_\Y+\|Ax_\alpha-A\hat{x}_\alpha \|_\Y \leq \delta + \|Ax-Ax_\alpha \|_\Y\] and the second for
\[ \|g^\delta-A\hat{x}_\alpha\|_\Y \leq \| (Ax-Ax_\alpha) - (g^\delta-A\hat{x}_\alpha)\|_\Y + \|Ax-Ax_\alpha\|_\Y \leq \delta + \|Ax-Ax_\alpha\|_\Y. \]
\end{proof} 
%As it fits in the context and we will need it later we state the following Lipschitz property. 
%\begin{lemma}\label{lem:Lipschitz} Let $g,h \in \Y$ and $\xh\in R_\alpha(g)$, $\hat{z}_\alpha\in R_\alpha(h)$. Then  \[ \|A\xh - A\hat{z}_\alpha\|_\Y \leq \|g-h\|_\Y. \] 
%\end{lemma} 
%\begin{proof}
%Using \Cref{prop:introduceQ} we have \[  \|A\xh - A\hat{z}_\alpha\|_\Y= \| \prox_{\alpha\mathcal{Q}}(g) -\prox_{\alpha\mathcal{Q}}(h)\|_\Y  \leq \|g-h\|_\Y \] 
%since firmly nonexpansivness implies Lipschitz continuity with constant $1.$
%\end{proof}
\begin{proposition}\label{lem:image_space_bounds} 
Let $\nu\in (0,1]$ and $\alpha, \varrho>0$. Suppose $x\in K_\nu^\varrho$. 
\begin{remunerate} 
\item Let $c_r>0$. If $\alpha\leq c_r  \varrho^{-\frac{1}{\nu}} \delta^\frac{1}{\nu}$ then \[ \|Ax-A\xh \|_\Y\leq (1+c_r^\nu) \delta \quad\text{for all}\quad \xh\in R_\alpha(g^\delta).\]  
\item Let $c_D>1$. If $\xh\in R_\alpha(g^\delta)$ satisfies $c_D \delta\leq \|g^\delta - A\xh \|$, then \[ (c_D-1)^\frac{1}{\nu}\varrho^{-\frac{1}{\nu}}  \delta^\frac{1}{\nu}  \leq \alpha.\] 
\end{remunerate} 
\end{proposition}
\begin{proof}
Let $x_\alpha \in R_\alpha(Ax).$
\begin{remunerate}
\item By \eqref{eq:imbound_Ax} and the definition of $\varrho_\nu$ we obtain 
\[ \|Ax - A\xh\|_\Y  \leq \delta + \ \varrho \alpha^\nu \leq (1+c_r^\nu) \delta.\] 
\item The bound \eqref{eq:imbound_g_delta} implies
\[ c_D\delta \leq \delta + \|Ax -A x_\alpha\|_\Y \leq \delta + \varrho \alpha^\nu. \] 
Subtracting $\delta$ and rearranging yields the claim. 
\end{remunerate}
\end{proof}
\subsection{Regularity of the minimizers}
First we recall the well-known fact that the source condition \eqref{eq:classical_sc} implies a linear convergence rate in the image space (see e.g. \cite[Lem.~3.5]{GHS:11}).
\begin{lemma} \label{prop:source_implies_linear}
Let \neww{$z\in \X$} and assume $\omega\in\Y$ with $A^\ast \omega\in \partial\mathcal{R}(z).$ Then 
\[ \|Az- A z_\alpha \|_\Y\leq \|\omega\|_\Y \alpha \quad\text{for all } \alpha>0 \text{ and } z_\alpha\in R_\alpha(Az) .\] 
 \end{lemma} 
\begin{proof}
The first order optimality condition yields ${\xi_\alpha:=\frac{1}{\alpha} A^\ast A (z-z_\alpha)\in \partial\mathcal{R}(z_\alpha)}$. Solving the inequality 
\begin{align*}
\frac{1}{\alpha}\|Az-Az_\alpha\|^2 = \langle \xi_\alpha, z-z_\alpha \rangle  \leq \mathcal{R}(z)-\mathcal{R}(z_\alpha) \leq \langle A^\ast \omega, z-z_\alpha \rangle  \leq  \|\omega\|_\Y \|Az-Az_\alpha \|_\Y
\end{align*} 
for $\|Az-Az_\alpha\|_\Y$ proves the claim.
\end{proof}
\begin{lemma}\label{prop:reg_of_mini}
Let $\alpha>0$, $\xh\in R_\alpha(g^\delta)$. Furthermore let $\beta>0$, $(\xh)_\beta\in R_\beta(A\xh)$ and $x_\beta\in R_\beta(Ax)$. 
\begin{remunerate}
\item If $\beta\in (0,\alpha]$ then \[ \| A\xh- A(\xh)_\beta\|_\Y \leq\frac{\beta \delta}{\alpha} +  \|Ax-Ax_\beta\|_\Y  .\] 
\item If $\beta\in [\alpha,\infty)$ then
 \[  \| A\xh- A(\xh)_\beta\|_\Y \leq \delta 
  + 2 \|Ax-Ax_\beta\|_\Y . \] 
\end{remunerate}
\end{lemma} 
\begin{proof} 
\begin{remunerate}
\item \neww{By the first order optimality condition the element $\xh$ satisfies the prerequisite $A^\ast \omega\in \partial\mathcal{R}(\xh)$ of \Cref{prop:source_implies_linear} with $\omega=\frac{1}{\alpha}(g^\delta -A \xh).$} \\
By \Cref{app:monotone} the map $\alpha \mapsto \frac{1}{ \alpha} \|Ax -Ax_\alpha \|_\Y$ is non increasing. Together with \eqref{eq:imbound_g_delta}  we obtain
\begin{align*}
\neww{\|\omega\|_\Y=} \frac{1}{\alpha} \|g^\delta -A \xh\|_\Y \leq \frac{\delta}{\alpha}  + \frac{1}{\alpha}\|Ax -Ax_\alpha \|_\Y\leq  \frac{\delta}{\alpha} + \frac{1}{\beta}\|Ax -Ax_\beta \|_\Y .
\end{align*}
Hence \Cref{prop:source_implies_linear} implies the claim.
\item We use first \Cref{prop:firmly} with $g=Ax$ and $h=A\xh$ then \eqref{eq:imbound_Ax} and finally non decreasingness of $\alpha\mapsto \|Ax-Ax_\alpha \|_\Y$ (see \Cref{app:monotone}) to estimate 
\begin{align*}
\|(Ax-A x_\beta) - (A\xh-A (\xh)_\beta)  \|_\Y&  \leq \|
Ax - A\xh\|_\Y \\ 
& \leq \delta +  \|Ax - Ax_\alpha\|_\Y \\ 
&\leq \delta +  \|Ax - Ax_\beta\|_\Y
\end{align*}
The triangle inequality finishes the proof. 
\end{remunerate}
\end{proof}
\begin{proposition}\label{lem:reg_of_mini}
Let $\nu\in (0,1]$ and $\varrho, c_l,\alpha>0$.
Suppose $x\in K_\nu^\varrho$ and \linebreak ${\xh\in R_\alpha(g^\delta)}$. If $c_l \varrho^{-\frac{1}{\nu}}  \delta^\frac{1}{\nu} \leq \alpha$, then $\varrho_\nu(\xh)\leq (2+c_l^{-\nu})\varrho$.
\end{proposition} 
\begin{proof}
Let $\beta\leq \alpha$. With $\delta\leq c_l^{-\nu} \varrho\alpha^\nu$ we estimate
\[ \frac{\delta \beta}{\alpha}\leq c_l^{-\nu} \varrho\alpha^{\nu-1} \beta \leq c_l^{-\nu} \varrho \beta^\nu .\] 
Furthermore 
\[ \delta\leq c_l^{-\nu} \varrho\alpha^\nu\leq c_l^{-\nu} \varrho\beta^\nu \quad\text{for all } \beta\geq \alpha.\]
Together with $\|Ax-Ax_\beta\|_\Y\leq \varrho \beta^\nu$ for all $\beta>0$ and $x_\beta\in R_\beta(Ax)$ the result follows from \Cref{prop:reg_of_mini}. 
\end{proof}
\subsection{Almost minimaxity on the sets  \texorpdfstring{$K_\nu$}{}}\label{sec:minimaxity_on}
Now we are in position to give the proof of \Cref{thm:minimax}.
\begin{proof}[Proof of \Cref{thm:minimax}]\label{pr:thm_minimax_bound}
\begin{remunerate}
\item By \Cref{lem:image_space_bounds} we have $\|Ax-A\xh\|_\Y\leq c_1 \delta$ and \Cref{lem:reg_of_mini} yields $x, \xh\in K_\nu^{c_2\varrho}$. 
\item Using the triangle inequality we obtain \[ \|Ax-A\xh\|_\Y\leq \delta + \|g^\delta - A\xh\|_\Y \leq d_1 \delta. \] 
\Cref{lem:image_space_bounds} provides $(c_D-1)^\frac{1}{\nu} \varrho^{-\frac{1}{\nu}}\delta^\frac{1}{\nu}\leq\alpha.$ Therefore \Cref{lem:reg_of_mini} yields \linebreak $x, \xh\in K_\nu^{d_2 \varrho}$.
\end{remunerate}
In both cases the claim follows from the definition of the modulus $\Omega.$
\end{proof}

\section{Convergence rates theory for Banach space regularization} 
\subsection{Source-wise representations and linear image space approximation} 
%the simple obervation that the source condition \eqref{eq:classical_sc} implies that $x$ is $\mathcal{R}$-minimal in $A\inv (\{Ax\})$. Note that this trivial if $A$ is injective. 
%\begin{lemma} \label{prop:finite_rho_implies_minimality}
%Suppose $x\in \XA$ satisfies \eqref{eq:classical_sc}. Then 
%\[ \mathcal{R}(x)=  \inf \{ \mathcal{R}(z) \colon z\in \XA \text{ with } Az=Ax\} < \infty.\] 
%\end{lemma} 
%\begin{proof}
%Let $\omega\in \Y$ such that $A^\ast \omega \in \partial \mathcal{R}(x).$ Then $\partial \mathcal{R}(x)\neq \emptyset$ implies $\mathcal{R}(x)<\infty$.  If $z\in \X$ with $Az=Ax$ then 
%$ \mathcal{R}(x)= \mathcal{R}(x) + \langle A^\ast \omega , z-x\rangle \leq \mathcal{R}(z).$   
%\end{proof}
%Recall that condition \eqref{eq:classical_sc} implies a linear bound of the form $\|Ax-Ax_\alpha\|_\Y \leq \|\omega\|_\Y \alpha$ where $\omega\in \Y$ is a source element, i.e. $A^\ast \omega \in \partial\mathcal{R}(x)$ (see ).
We start with a converse to \Cref{prop:source_implies_linear}: A linear bound $\|Ax-Ax_\alpha\|_\Y=\mathcal{O}(\alpha)$ implies the source condition \eqref{eq:classical_sc} and the minimal $\mathcal{O}$-constant $\varrho_1(x)$ agrees with the minimal norm $\|\omega\|_\Y$ attended by a source element $\omega$. Similar results can be found in \cite[Lem.~4.1]{GHS:11} and \cite[Prop. 4.1]{RR:10}. For sake of self-containedness we include a proof.
%Recall that a source-wise representation of the form \eqref{eq:classical_sc} is equivalent to a linear image space bound $\|Ax - Ax_\alpha\|_\Y = \mathcal{O}(\alpha)$, i.e. $K_1$ is the set of all $x\in \X$ satisfying \eqref{eq:classical_sc} (see \Cref{lem:characterization_linear_case} or \cite[Lem.~4.1]{GHS:11}, \cite[Prop. 4.1]{RR:10}).\\
\begin{proposition} \label{lem:characterization_linear_case}
Let \neww{$x\in \X$} with $\mathcal{R}(x)=  \inf \{ \mathcal{R}(z) \colon z\in \X \text{ with } Az=Ax\}$. Then 
 \[ \varrho_1(x)= \inf\left\{ \|\omega\|_\Y \colon A^\ast\omega \in \partial\mathcal{R}(x) \right\}. \] 
If this quantity is finite and $x_\alpha\in R_\alpha(Ax)$, $\alpha>0$ is any selection, then the net \linebreak $(\frac{1}{\alpha} (Ax-Ax_\alpha))_{\alpha>0}$ convergences weakly for $\alpha \searrow 0$ to the unique $\omega\in \Y$ with $A^\ast \omega \in \partial\mathcal{R}(x)$ and $\|\omega\|_\Y=\varrho_1(x).$ 
\end{proposition}
\begin{proof}Taking the infimum over $\omega$ in 
\Cref{prop:source_implies_linear} yields 
\[  \varrho_1(x)\leq  \inf\left\{ \|\omega\|_\Y \colon A^\ast\omega \in \partial\mathcal{R}(x) \right\}.\]
To prove the remaining inequality let \neww{$x\in \X$} with $\varrho_1(x)<\infty$. \neww{Then the net\linebreak ${(\frac{1}{\alpha} (Ax-Ax_\alpha))_{\alpha>0}}$ is norm bounded in the Hilbert space $\Y.$} By the 
 Banach–Alaoglu theorem every null sequence of positive numbers has a subsequence $\alpha_n>0$ such that $\frac{1}{\alpha_n}(Ax-Ax_{\alpha_n})$ converges weakly to some $\omega\in \Y$ with $\|\omega\|_\Y\leq \varrho_1(x).$ \Cref{app:minimal_limit} and the minimality assumption yield  \[ \frac{1}{2\alpha_n} \|Ax-Ax_{\alpha_n}\|_\Y^2 + \mathcal{R}(x_{\alpha_n}) \rightarrow \mathcal{R}(x). \]
Together with $\|Ax-Ax_{\alpha_n}\|_\Y \leq \varrho_1(x) \alpha_n$ we obtain 
$ \mathcal{R}(x_{\alpha_n})\rightarrow \mathcal{R}(x).$
The first order optimality  condition yields $\frac{1}{\alpha_n} A^\ast A(x-x_{\alpha_n}) \in \partial\mathcal{R}(x_{\alpha_n})$. Hence for \neww{$z\in \X$} we obtain 
\begin{align*}
\mathcal{R}(x)+ \langle A^\ast \omega , z-x \rangle &  = \mathcal{R}(x)+ \langle \omega , A(z-x) \rangle \\ 
& = \lim_{n\rightarrow \infty} \mathcal{R}(x_{\alpha_n}) + \langle \frac{1}{\alpha_n}A(x-x_{\alpha_n}) , A(z-x_{\alpha_n}) \rangle  \\ 
& = \lim_{n\rightarrow \infty} \mathcal{R}(x_{\alpha_n}) + \langle\frac{1}{\alpha_n} A^\ast A(x-x_{\alpha_n}) , z-x_{\alpha_n} \rangle \leq \mathcal{R}(z).
\end{align*}
This shows $A^\ast \omega\in \partial\mathcal{R}(x)$. Therefore the stated identity is proven.\\
Being the the preimage of the convex set $\partial\mathcal{R}(x)$ under the linear map $A^\ast$ the set $\{\omega\in \Y \colon A^\ast \omega\in \partial\mathcal{R}(x)\}$ is convex. Strict convexity of $\|\cdot \|_\Y$ yields uniqueness of $\omega.$ In particular this implies the convergence of the net.
\end{proof}

\begin{corollary}\label{cor:definitness_of_rholin} We have $\varrho_1 (x) = 0$ if and only if $x\in \argmin_{z\in\neww{\X}}\mathcal{R}(z).$
\end{corollary} 
\begin{proof}
By the second statement in \Cref{lem:characterization_linear_case} we have  $\rho_{\neww{1}} (x) = 0$ if and only if $0\in \partial\mathcal{R}(x)$. Hence the first order optimality condition
$x\in \argmin_{z\in\X} \mathcal{R}(z)$ if and only if $0\in \partial\mathcal{R}(x)$ yields the claim.
\end{proof}
\begin{example}\label{ex:easyvarrho}
Let \neww{$p \in [1,2]$, $\X=\ell^p :=\ell^p(\mathbb{N})$}, $  \Y= \ell^2:= \ell^2(\mathbb{N})$, \neww{$A\colon \ell^p\rightarrow \ell^2$ the embedding operator given by $x\mapsto x$}  and $\mathcal{R}$ given by $\mathcal{R}(x)= \frac{1}{p} \|x\|_{\ell^p}$.
%We have $\partial\mathcal{R}(x)= \partial\left(\mathcal{R}|_{\ell^p} \right)(x) \cap \ell^2$. \\
Let $x\in \ell^p$.\\ 
If $p>1$ then  \neww{$\partial\mathcal{R}(x)=\{ \xi \}$} with $|\xi_j|= |x_j|^{p-1}$. \neww{The adjoint $A^\ast$ identifies with the embedding operator $\ell^2\rightarrow \ell^{p^\prime}$ with $p^\prime$ the H\"older conjugate of $p$.}
 Hence $x\in K_1$ if and only if $\|\xi\|_{\ell^2}<\infty$, and we have \[ \varrho_1(x)= \|\xi\|_{\ell^2}= \left( \sum\nolimits_{j\in \mathbb{N}} |x_j|^{2p-2}  \right)^{1/2} = \|x \|_{2p-2}^{p-1}. \]
Therefore \Cref{ass:K_1_quasi_banach} is satisfied in this case. \\
For $p=1$ we have \neww{$\xi \in \partial\mathcal{R}(x)$} if and only if $\xi_j=1$ for $x_j >0$, $\xi_j=-1$ for $x_j <0$ and $|\xi_j| \leq 1$ for $x_j=0$. Hence $K_1$ consists of all elements with finitely many non vanishing coefficients. We have 
$ \varrho_1(x)= \# \left\{ j\in \mathbb{N} \colon x_j \neq 0 \right\}^{1/2}$ and \Cref{ass:K_1_quasi_banach} is not fulfilled.    
\end{example}
\subsection{Computation of \texorpdfstring{$K_1$}{} for Banach space regularization}
In this subsection we assume $\XA$ is a Banach space with a dense, continuous embedding $\X\subset \XA$ and that $A$ extends to $\XA$ such that \eqref{eq:XAequiA} is satisfied. 
%Note that injectivity is necessary for \eqref{eq:XAequiA}.\\
Let $u\in [1,\infty)$ and consider the penalty given by $\mathcal{R}(x)=\frac{1}{u}\|x\|_\X^u$.\\
%Note that starting with an injective, bounded linear operator $A\colon \XR\rightarrow \Y$ there always exists $\XA$ with a continuous dense embedding $\XR\subset \XA$ such that $A$ extends to $\XA$ fulfilling \eqref{eq:XAequiA} (take the Banach completion of $\XR$ in the norm $x\mapsto \|Ax\|_\Y$).\\
\neww{If $\X$ is reflexive we choose $\tau$ to be the weak topology on $\X$. Then the sublevel sets of $\mathcal{R}$ are $\tau$-compact by the Banach–Alaoglu theorem.
Moreover $A$ is weak-to-weak continuous as it is bounded. Therefore, \Cref{ass:topological} is automatically satisfied in this case.}\\
In this subsection we compute $K_1$ for \neww{the three penalties covered in the examples in \Cref{sec:main_results_banach}}. We start with a tool that helps computing the function $\varrho_1$ up to equivalence. Note that the density $\X\subset \XA$ allows us to view the adjoint of the embedding as an embedding $\X_A ^\prime \subset \X^\prime$. 
\begin{proposition}\label{prop:computation_rholin}
\neww{We have $x\in K_1$ if and only if $\partial
\mathcal{R}(x)\cap \X_A ^\prime\neq \emptyset.$}
The function 
\begin{align*}
\overline{\varrho}_1 \colon \XA \rightarrow [0,\infty]  \quad\text{given by}\quad \overline{\varrho}_1(x)= \inf \left\{\|\xi\|_{\X_A^\prime} \colon \xi\in\partial
\mathcal{R} (x) \cap \X_A^\prime \right\}
\end{align*}
satisfies 
\[   \frac{1}{M} \varrho_1 (x) \leq \overline{\varrho}_1(x) \leq M \varrho_1 (x) \quad\text{for all } x\in \neww{\X}.\]
\end{proposition}
\begin{proof}
Suppose $\xi\in\partial
\mathcal{R} (x)\cap \X_A^\prime$. Let $z\in \XA$, then 
\[ \langle \xi, z\rangle \leq \|\xi\|_{\X_A^\prime} \|z\|_{\X_A} \leq M \|\xi\|_{\X_A^\prime} \|Az\|_\Y.\] 
\Cref{lem:boundedness_linear_fct_range} provides $\omega\in \Y$ with $\|\omega\|_\Y\leq M \|\xi\|_{\XA^\prime} $ and $A^\ast \omega = \xi\in \partial\mathcal{R}(x)$. Together with \Cref{lem:characterization_linear_case} this yields the first inequality. \\
Let $\omega\in \Y$, such that $A^\ast \omega \in \partial
\mathcal{R}(x)$. Then\[ \langle A^\ast \omega ,z \rangle = \langle \omega , Az \rangle \leq \|\omega \|_{\Y} \|Az\|_\Y \leq  M   \|\omega \|_{\Y}  \|z\|_{\X_A} \] 
for all $z\in \X.$ Hence $ \| A^\ast \omega \|_{\X_A^\prime}\leq M \|\omega\|_\Y $. This proves the second inequality. 
\end{proof}
\subsubsection*{Computation of \texorpdfstring{$K_1$}{}  for weighted \texorpdfstring{$p$}{}-norm penalization}\label{sec:wei}
We revisit the first example in \Cref{sec:main_results_banach}. Recall $\XA=\lspace  {\overline{a}} 2$ and $\neww{\X}=\lspace {\overline{r}}p$ with $p\in (1,2)$.
\begin{proposition}\label{prop:varrho_1_wei}
Let $\overline{s}=\overline {a}^{-\frac{1}{p-1}}\overline{r}^\frac{p}{p-1}$. Then $K_1= \lspace {\overline{s}} {2p-2}$ with
\[ \frac{1}{M} \varrho_1(x)\leq  \lnorm x s {2p-2} ^ {p-1} \leq M \varrho_1(x) \quad\text{for all } x\in \lspace  {\overline{a}}2.\] 
\end{proposition} 
\begin{proof}
Let $x\in \lspace {\overline{r}} p$. Then $\neww{\partial \mathcal{R}}(x)=\{\xi\}$ with $|\xi_j|=  r_j^p |x_j|^{p-1}$.
With $\overline{\varrho}_1$ \neww{as in \Cref{prop:computation_rholin} and in view of} \Cref{app:dual} we obtain 
\[ \overline{\varrho}_1(x)= \lnorm \xi {a\inv} 2 =  \left( \sum\nolimits_{j\in\Lambda} a_j^{-2} r_j^{2p} |x_j|^{2p-2} \right)^{1/2}= \lnorm x s {2p-2} ^ {p-1}. \] 
\Cref{prop:computation_rholin} yields the result. 
\end{proof}
\subsubsection*{Computation of \texorpdfstring{$K_1$}{} for Besov \texorpdfstring{$0,p,p$}{}-penalties}\label{sec:besov0pp}
Next we characterize $K_1$ for Example $2.$ Recall $\XA=\bspace {-a}22$ and $\neww{\X}=\bspace 0pp$ with $p\in (1,\infty)$.
\begin{proposition} \label{prop:varrho_1_besov_pp}
Let $\tilde{s}=\frac{a}{p-1}$ and $\tilde{t}= 2p-2$. Then $K_1= \bspace {\tilde{s}} {\tilde{t}}{\tilde{t}}$ with 
\[ \frac{1}{M} \varrho_1(x)\leq \bn x{\tilde{s}} {\tilde{t}}{\tilde{t}} ^{p-1} \leq M \varrho_1(x) \quad\text{for all } x\in \bspace {-a}22. \] 
\end{proposition} 
\begin{proof}
The proof works along the lines of the proof of \Cref{prop:varrho_1_wei} by identifying the expression for $\bn \xi a22$ with $\bn x{\tilde{s}} {\tilde{t}}{\tilde{t}}$.
\end{proof}
\subsubsection*{Computation of \texorpdfstring{$K_1$}{} for Besov \texorpdfstring{$0,2,q$}{}-penalties} \label{sec:besov02q}
Finally we compute $K_1$ for Example $3.$ with $\XA=\bspace {-a}22$ and $\neww{\X}=\bspace 02q$ with $q\in (1,\infty)$.
\begin{proposition}\label{prop:varrho_1_besov_2q} 
Let $\tilde{s}=\frac{a}{q-1}$ and $\tilde{q}=2q-2$. Then $K_1=\bspace {\tilde{s}}2{\tilde{q}}$ with 
\[ \frac{1}{M} \varrho_1(x)\leq \bn x{\tilde{s}} {2}{\tilde{q}} ^{q-1} \leq M \varrho_1(x) \quad\text{for all } x\in \bspace {-a}22. \]
%\item For $q=1$ the set $K_1$ consists of all sequences with only finitly many non vanishing enties and we obtain \[ \frac{1}{M} \varrho_1(x)\leq \left( \sum_{j \in \mathbb{N}_0} 2^{2ja} \mathds{1}_{\{x_j \neq 0 \} } \right)^\frac{1}{2} \leq M \varrho_1(x) \quad\text{for all } x\in \bspace {-a}22. \]
\end{proposition} 
\begin{proof}
 If $x\in \bspace 02q$, then $\neww{\partial\mathcal{R}}(x)= \{ \xi\}$ with $\xi_{j,k}=\left(\sum_{k^\prime} |x_{j,k^\prime} |^2 \right)^{\frac{q}{2}-1} | x_{j,k}|$. With $\overline{\varrho}_1$ \neww{as in \Cref{prop:computation_rholin} and using \Cref{app:dual}} we obtain   
${ \overline{\varrho}_1(x)= \bn \xi a22 = \bn x {\tilde{s}} 2 {\tilde{q}} ^{q-1}.}$
%\item Using Asplund's theorem (see 
%\cite[I.\S 4, Thm.~4.4]{cioranescu:90}) we have $\xi\in \partial\mathcal{R}(x)$ if and only if $\|\xi_j\|_2\leq 1$ for all $j\in \mathbb{N}_0$ and $\xi_j= \frac{x_j}{\|x_j\|_2}  $ for all $j\in \mathbb{N}_0$ with $x_j\neq 0$. Therefore the infimum in \eqref{eq:def_rholino} is attained by $\xi\in \mathcal{R}(x)$ with $\xi_j=0$ for all $\in \mathbb{N}_0$ with $x_j= 0$.
\Cref{prop:computation_rholin} yields the result.
\end{proof}
Note that \Cref{ass:K_1_quasi_banach} holds true for all three examples.

\subsection{Characterizations of \texorpdfstring{$K_\nu$}{}}
\subsubsection*{\texorpdfstring{$K_\nu$}{} via approximation by elements of \texorpdfstring{$K_1$}{} } \label{sec:gamma}
In \cite[Prop.~1]{BO:04} the authors point out that the set of elements satisfying the source condition \eqref{eq:classical_sc} is the set of possible minimizers of the Tikhonov functional. Therefore one might suggest that the approximation error of $\neww{x\in \X}$ by $x_\alpha\in R_\alpha(Ax)$ is determined by the best  approximation from the family of sets \[B_r:=\{ x\in \X \colon \varrho_1 (x) \leq r\} \quad\text{ with } r\geq 0 .\] We consider the best approximation error 
\[ \gamma_x \colon [0,\infty) \rightarrow [0,\infty) \quad\text{given by } \gamma_x(r)= \inf_{z\in B_r} \|Ax-Az\|_\Y.\] 
The function $\gamma_x$ is well defined as \Cref{cor:definitness_of_rholin} yields  
$\emptyset\neq \argmin_{z\in \neww{\X} }\mathcal{R}(z)\subset B_r$ for all $r\geq 0$. Moreover it is non increasing as $B_{r_1}\subseteq B_{r_2}$ for $r_1\leq r_2$. \\
The following proposition is the starting point to prove equivalence of H\"older-type bounds on  $\gamma_x$ and on $\|Ax-Ax_\alpha\|_\Y$.
\begin{proposition}\label{lem:approx_image_tik}  
Let \neww{$x\in \X$}, $\alpha>0$ and $x_\alpha\in R_\alpha(Ax)$. Then 
\[ \gamma_x\left( \frac{1}{\alpha} \|Ax-Ax_\alpha\|_\Y\right)\leq  \|Ax-Ax_\alpha\|_\Y \leq  4\gamma_x\left( \frac{1}{4 \alpha} \|Ax-Ax_\alpha\|_\Y\right) . \] 
\end{proposition}  
\begin{proof}
\neww{\Cref{lem:characterization_linear_case}} and the first order \neww{optimality} condition  ${\frac{1}{\alpha} A^\ast A (x-x_\alpha)}\in \partial\mathcal{R}(x_\alpha)$ provide ${\varrho_1(x_\alpha) \leq \frac{1}{\alpha}\|Ax-Ax_\alpha\|_\Y}$. This proves the first inequality by definition of $\gamma_g$.\\
To \neww{show} the second inequality let $z\in B_r$. By 
\Cref{lem:characterization_linear_case} there is $\omega\in \Y$ with $\|\omega\|_\Y\leq r$ and $A^\ast \omega\in \partial\mathcal{R}(z)$ hence \[ \mathcal{R}(z)-\mathcal{R}(x_\alpha)\leq \langle A^\ast \omega , z-x_\alpha \rangle \leq r \|Az-Ax_\alpha\|_\Y.\] 
From $2\alpha T_\alpha(x_\alpha,Ax)\leq 2\alpha T_\alpha(z,Ax)$ and the last inequality we deduce 
\begin{align*}
\|Ax-Ax_\alpha \|_\Y^2 & \leq \|Ax-Az \|_\Y^2 + 2\alpha r \|Az-Ax_\alpha\|_\Y \\ & \leq  \|Ax-Az \|_\Y^2 + 2\alpha  r  \|Ax-Az\|_\Y + 2\alpha r \|Ax-Ax_\alpha\|_\Y  
\end{align*}
Taking the infimum over $z\in B_r$ and estimating the third summand using \linebreak ${ab\leq \frac{1}{2} a^2 + \frac{1}{2} b^2}$ we obtain   
\begin{align*}
\|Ax-Ax_\alpha \|_\Y^2 & \leq \gamma_x(r)^2 + 2\alpha  r  \gamma_x(r) + 2 \alpha^2 r^2 + \frac{1}{2} \|Ax-Ax_\alpha \|_\Y^2   \\ 
& \leq 2 (\gamma_x(r)+\alpha  r )^2 + \frac{1}{2} \|Ax-Ax_\alpha \|_\Y^2.
\end{align*}
Hence $ \|Ax-Ax_\alpha \|_\Y \leq 2 \gamma_x(r)+2 \alpha  r$ and the choice $r=\frac{1}{4\alpha}\|Ax-Ax_\alpha \|_\Y$ yields the second inequality. 
\end{proof}
As announced we see equivalence of H\"older-type bounds on  $\gamma_x$ and on $\|Ax-Ax_\alpha\|_\Y$ as a consequence.
\begin{proposition} \label{prop:image_bound_vs_approx_bound}
Let $\nu\in (0,\infty)$ and $\neww{x\in \X}$.
The following statements are equivalent: 
\begin{romannum}
\item There exists a constant $c_1>0$ such that $\gamma_x(r)\leq c_1 r^{-\nu}$ for all $r>0$. 
\item There exists a constant $c_2>0$ such $\|Ax-Ax_\alpha\|_\Y\leq c_2 \alpha^\frac{\nu}{1+\nu}$ for all $\alpha>0$ and $x_\alpha\in R_\alpha(Ax).$
\end{romannum}
\end{proposition}
More precisely $(i)$ implies $(ii)$ with $c_2=4 c_1^\frac{1}{1+\nu}$ and $(ii)$ implies $(i)$ with $c_1= c_2^{1+\nu}.$
\begin{proof}
\item[(i)$\Rightarrow$(ii):]  
The second inequality in \Cref{lem:approx_image_tik} yields 
\[ \|Ax-Ax_\alpha\|_\Y\leq 4^{1+\nu}  c_1 \alpha^\nu  \|Ax-Ax_\alpha\|_\Y^{-\nu}\] 
Multiplying by $\|Ax-Ax_\alpha\|_\Y^{\nu}$ and taking the power $\frac{1}{1+\nu}$ yields \[ \|Ax-Ax_\alpha\|_\Y\leq 4 c_1^\frac{1}{1+\nu} \alpha^\frac{\nu}{1+\nu}\]
\item[(ii)$\Rightarrow$(i):] 
Let $r>0.$ For $\alpha=c_2^{1+\nu} r^{-(1+\nu)}$ we obtain 
$ \frac{1}{\alpha}\|Ax-Ax_\alpha\|_\Y\leq c_2 \alpha^{-\frac{1}{1+\nu}} =r.$ 
Hence the first inequality in \Cref{lem:approx_image_tik} yields
 \[ \gamma_x(r)\leq \gamma_x \left(\frac{1}{\alpha}\|Ax-Ax_\alpha\|_\Y \right) \leq \|Ax-Ax_\alpha\|_\Y\leq c_2 \alpha^\frac{\nu}{1+\nu} = c_2^{1+\nu}r^{-\nu}.\]
\end{proof}

\subsubsection*{\texorpdfstring{$K_\nu$}{} via real interpolation}
Again we assume $\XA$ is a Banach space such that \eqref{eq:XAequiA} holds true.
The next lemma shows that under \Cref{ass:K_1_quasi_banach} the spaces $\left(\XA,K_1\right)_{\theta,\infty}$  classify the image space approximation precision. 
\begin{proposition}[$K_\nu$ as a real interpolation space] \label{prop:K_nu_under_quasi_assumption}
Suppose \Cref{ass:K_1_quasi_banach} holds true. Let $\theta\in (0,1)$ and $\nu:=\frac{\theta}{(1-\theta)(u-1)+\theta}$. 
We have $K_\nu=(\XA, K_1)_{\theta,\infty}$ with 
\[ C_1 \|x\|_{\left(\XA,K_1\right)_{\theta,\infty}} \leq \varrho_\nu (x)^\frac{(1-\theta)(u-1)+\theta}{u-1}\leq C_2  \|x\|_{\left(\XA,K_1\right)_{\theta,\infty}} \quad\text{for all } x\in \left(\XA,K_1\right)_{\theta,\infty}\]
with constants $C_1$, $C_2 > 0$ depending only on $u, \theta$ and $M$.  
\end{proposition} 
\begin{proof}
Assume $\varrho:= \varrho_\nu (x) < \infty.$ \Cref{prop:image_bound_vs_approx_bound} provides the bound \[ \gamma_x(r)\leq \varrho^{\frac{(1-\theta)(u-1)+\theta}{(1-\theta)(u-1)}} r^{-\frac{\theta}{(1-\theta)(u-1)}} \]
Let $t>0$. We choose $r:=\varrho^{(1-\theta)(u-1)+\theta} t^{-(1-\theta)(u-1)}$. 
If $\varepsilon>0$ then there exists $z\in K_1$ with $\varrho_1 (z)\leq r$ and $\|Ax-Az\|_\Y\leq \gamma_x(r) +\varepsilon.$ Therefore we obtain 
\begin{align*}
K(x,t)\leq \|x-z\|_\XA + t \|z\|_{\textrm{lin}} & \leq M (\gamma_x(r) +\varepsilon) + t M^\frac{1}{u-1} r^\frac{1}{u-1}. 
\end{align*}
For $\varepsilon\rightarrow 0$ we obtain 
\[ K(x,t)\leq M \varrho^{\frac{(1-\theta)(u-1)+\theta}{(1-\theta)(u-1)}} r^{-\frac{\theta}{(1-\theta)(u-1)}} +  t M^\frac{1}{u-1} r^\frac{1}{u-1} = \left(M+M^\frac{1}{u-1} \right) \varrho^\frac{(1-\theta)(u-1)+\theta}{u-1} t^\theta.\] 
This proves the first inequality. \\
Assume $n:= \|x\|_{\left(\XA,K_1\right)_{\theta,\infty}} <\infty$. We prove a bound on $\gamma_x$ and apply \Cref{prop:image_bound_vs_approx_bound}. Let $r>0$. We choose $t:= 2 M^{\frac{1}{(1-\theta)(u-1)}} n^\frac{1}{1-\theta} r^{-\frac{1}{(1-\theta)(1-u)}}$. Since $2^{1-\theta}>1$ there exists $z\in \XA$ such that
 \begin{align*}
  M\inv \|Ax-Az\| + t M^{-\frac{1}{u-1}} \varrho_1(z) ^\frac{1}{u-1}  & \leq \|x-z\|_\XA + t \|z\|_{\textrm{lin}}  \\ 
  &\leq 2^{1-\theta} K(x,t)  \leq 2^{1-\theta} n t^\theta = t M^{-\frac{1}{u-1}} r^\frac{1}{u-1}
\end{align*}   
Neglecting the first summand on the left hand side we obtain $\varrho_1(z)\leq r$. Therefore 
\[\gamma_x (r) \leq \|Ax-Az\|_\Y \leq  M 2^{1-\theta} n t^\theta = 2 M^{\frac{(1-\theta)(u-1)+\theta}{(1-\theta)(u-1)}} n^\frac{1}{1-\theta} r^{-\frac{\theta}{(1-\theta)(u-1)}}.\]
\Cref{prop:image_bound_vs_approx_bound} yields 
$ \varrho_\nu (x)\leq 8M n^\frac{u-1}{(1-\theta)(u-1)+\theta} $.
\end{proof}
\begin{remark}
As already exposed in \Cref{ex:easyvarrho} we cannot expect \Cref{ass:K_1_quasi_banach} to hold true for $\ell^1$-type norms like Besov $0,1,1$ or $0,2,1$-norms. Nevertheless one may use \Cref{prop:image_bound_vs_approx_bound} directly to characterize the sets $K_\nu$ in this case. Applying \Cref{thm:minimax} then reproduces the convergence rates results for the $0,2,1$-penalty in \cite{HM:19} and for weighed $\ell^1$-penalties in \cite{MH:20} in the case of linear operators. 
\end{remark}
\subsection{Error bounds}\label{sec:error_Banach}
We apply \Cref{thm:minimax} to obtain error bounds \neww{measured in the norm of certain Banach spaces $\XL$ with a continuous embedding $\XL\subset \XA$}.\\
\neww{To this end we consider the loss function $L\colon \X \times\X\rightarrow [0,\infty]$ given by $L(x_1,x_2)=\|x_1-x_2\|_{\XL}$ if $x_1-x_2\in \XL$ and $L(x_1,x_2)=\infty$ if $x_1-x_2\notin \XL$}
\neww{Before we prove \Cref{thm:error_bounds}} we state a proposition that characterizes for which spaces $\XL$ H\"older-type bounds on the modulus of continuity on balls of a given \neww{quasi-}Banach space $\X_S\subset \X$ are satisfied. 
\begin{proposition}[bound on the modulus]\label{prop:bound_on_modulus}
Let $\X_S\neww{\subset \X}$ be \neww{a} quasi-Banach space and $\X_L$ a Banach space with continuous embeddings $\X_S\subset \X_L\subset \XA$ and $e\in (0,1)$. For $\varrho>0$ we denote \[K_{\X_S}^\varrho:= \left\{ x\in \X_S\colon \|x\|_{\X_S} \leq \varrho \right\}.\] The following statements are equivalent: 
\begin{romannum}
\item There is a continuous embedding $\left(\XA, \X_S\right)_{e,1}\subset \X_L$. 
\item There exists a constant $c>0$ with $ \Omega(\delta, K_{\X_S}^\varrho)\leq c \varrho^{e} \delta^{1-e}$ for all $\delta, \varrho>0.$
\end{romannum}
\end{proposition} 
\begin{proof}
By \cite[Sec.~3.5, Thm.~3.11.4]{Bergh1976} statement $(i)$ is equivalent to an interpolation inequality  
\begin{align}\label{eq:interpolation_modulus_bound}
\|z\|_{\XL} \leq C \|z\|_{\XA}^{1-e} \|z\|_{\X_S} ^e \quad\text{for all} \quad z\in \X_S .
\end{align}
Let $x_1, x_2\in K_{\X_S}^\varrho$ with $\|Ax_1- Ax_2\|_\Y\leq \delta$. The quasi-triangle inequality yields \linebreak $\|x_1-x_2\|_{\X_S}\leq 2c\varrho$  and from \eqref{eq:XAequiA} we obtain $\|x_1 - x_2 \|_{\XA} \leq M\delta$. Hence \eqref{eq:interpolation_modulus_bound} with $z= x_1-x_2$ yields  $\|x_1-x_2\|_\XL\leq  C M^{1-e} (2c)^e \varrho^e \delta^{1-e}$. Taking the supremum over $x_1, x_2$ yields $(ii)$. \\ 
Assuming a bound on the modulus we obtain \eqref{eq:interpolation_modulus_bound} from
\begin{align*}
\|z\|_\XL\leq \Omega \left(M\|z\|_\XA, K_{\X_S}^{\|z\|_{\X_S}}\right).
\end{align*}
\end{proof}
Next we give the proof of \Cref{thm:error_bounds}.
%\begin{theorem}[error bounds] 
%Suppose \Cref{ass:K_1_quasi_banach} holds true. 
%Let $0<\xi< \theta <1$ and $\delta,\varrho,\alpha>0$ and $c_r\geq c_l >0$, $C_D>c_D>1$. Suppose $\XL$ is a Banach space with a continuous embeddings $\left(\XA, K_1 \right)_{\xi,1}\subset \XL\subset \XA$. Assume $x\in \left(\XA, K_1 \right)_{\theta,\infty}$ with $\|x\|_{\left(\XA, K_1 \right)_{\theta,\infty}}\leq \varrho$ and $g^\delta\in \Y$ with $\|g^\delta - Ax \|_\Y\leq \delta$. Let $\xh\in R_\alpha(g^\delta)$. There exists a constant $C>0$ independent of $x,\delta$ and $\varrho$ such that whenever $\alpha$ satisfies either 
%\[ c_l \varrho^{-\frac{u-1}{\theta}} \delta^\frac{(1-\theta)(u-1)+\theta}{\theta} \leq \alpha \leq c_r \varrho^{-\frac{u-1}{\theta}} \delta^\frac{(1-\theta)(u-1)+\theta}{\theta}  \quad\text{or} \quad 
% c_D\delta\leq \|g^\delta- A\xh \|_\Y \leq  C_D \delta \] 
%the bound 
%\[ \|x-\xh \|_\XL \leq C \varrho^\frac{\xi}{\theta} \delta^{1-\frac{\xi}{\theta}} \] 
%holds true.
%\end{theorem}
\begin{proof}[Proof of \Cref{thm:error_bounds}]\label{pr:thm_error_Banach}
For $\nu$ as in \Cref{prop:K_nu_under_quasi_assumption} the second inequality therein yields 
\[x\in  K^{\varrho}_{(\XA,K_1)_{\theta,\infty}} \subset  K_\nu^{\overline{\varrho}}\]
with $\overline{\varrho}=\left(C_2 \varrho\right)^\frac{u-1}{(1-\theta)(u-1)+\theta}$.\\
In view of \Cref{thm:minimax} it remains to prove an upper bound on $\Omega(c_1\delta, K_\nu ^{c_2 \overline{\varrho}}) \leq C \varrho^\frac{\xi}{\theta} \delta^{1- \frac{\xi}{\delta}}$ for constants $c_1, c_2>0$ given therein. 
The first inequality in \Cref{prop:K_nu_under_quasi_assumption} provides  
\[ K_\nu ^{c_2 \overline{\varrho}}\subset K^{c_3\varrho}_{(\XA,K_1)_{\theta,\infty}} \quad\text{with}\quad c_3=C_1\inv C_2 c_2^\frac{(1-\theta)(u-1)+\theta}{u-1}.\]  
The reiteration theorem (see \cite[Thm.~3.11.5]{Bergh1976}) yields  
\begin{align}\label{eq:reiteration}
 \left(\XA,K_1\right)_{\xi,1}= \left( \XA , \left(\XA,K_1 \right)_{\theta,\infty}\right)_{\frac{\xi}{\theta},1}  
\end{align}
with equivalent quasi-norms. In particular $\left(\XA,K_1\right)_{\theta,\infty}\subset \left(\XA,K_1\right)_{\xi,1} \subset \XL \subset \XA$. Hence \Cref{prop:bound_on_modulus} with $\X_S=\left(\XA,K_1\right)_{\theta,\infty}$ yields a constant $c_4$ with 
\begin{align*}
\Omega\left(c_1 \delta, K_\nu^{c_2 \overline{\varrho}}\right)\leq  
\Omega\left(c_1 \delta, K_{\left(\XA,K_1\right)_{\theta,\infty}}^{c_3 \varrho} \right)\leq C \varrho^\frac{\xi}{\theta} \delta^{1-\frac{\xi}{\theta}} 
\end{align*}
with $C=c_4 c_3^\frac{\xi}{\theta} c_1^{1-\frac{\xi}{\theta}}$.\\
\neww{For the discrepancy principle the bound $\Omega(d_1\delta, K_\nu ^{d_2 \overline{\varrho}}) \leq C \varrho^\frac{\xi}{\theta} \delta^{1- \frac{\xi}{\delta}}$ follows by replacing $c_1$ by $d_1$ and $c_2$ by $d_2$.}
\end{proof}
\begin{remark}
The statement in \Cref{rem:limiting_case} for the limiting case $\theta=1$ follows along the same lines leaving out the step involving the reiteration theorem.  
\end{remark}
\begin{remark}
The relation $\left(\XA,\neww{K_1}\right)_{\xi,1}\subset \neww{\XL}$ is necessary to obtain error bounds as in \Cref{thm:error_bounds} in the following sense: Assuming $\XL$ satisfies an error bound 
\[ \|x-\xh \|_\XL \leq C \varrho^e \delta^{1-e} \] for some $e\in (0,1)$ and all $x\in K_{(\XA,K_1)_{\theta,\infty}}^\varrho$ under some apriori parameter choice $\alpha=\alpha(\delta)$, then the lower bound \eqref{eq:lower_bound_modulus} yields 
\[ \frac{1}{2}\Omega(2\delta, K_{(\XA,K_1)_{\theta,\infty}}^\varrho) \leq \Delta_{R_{\alpha(\delta)}} (\delta,K_{(\XA,K_1)_{\theta,\infty}}^\varrho )\leq C \varrho^e \delta^{1-e}. 
\]
Thus the converse implication in \Cref{prop:bound_on_modulus} and the identity \eqref{eq:reiteration} provides \[\left(\XA, K_1 \right)_{\theta e ,1 } =  \left(\XA, \left(\XA,K_1\right)_{\theta,\infty} \right)_{e,1}\subset \neww{\XL.} \]  
\end{remark} 

\subsubsection*{Error bounds for weighted \texorpdfstring{$p$}{}-norm penalization} 
To prove \newww{\Cref{thm:rates_weighted}} we return to the setting of Example $1.$ 
\begin{proof}[Proof of \Cref{thm:rates_weighted}]\label{pr:weighed}
First note that \Cref{ass:K_1_quasi_banach} holds true by \Cref{prop:varrho_1_wei}. By \cite[Thm.~2, Rem.]{F:78} we have 
%\label{eq:weighted_interpolation_error}
\begin{align*} 
 \lspace {\overline{r}} p  = \left( \lspace {\overline{a}}2 ,\lspace {\overline{s}} {2p-2} \right)_{\xi, p}=\left( \XA ,K_1 \right)_{\xi, p}  \quad\text{with}\quad \xi:=\frac{p-1}{p}. 
\end{align*}
Hence by \cite[Thm.~ 3.4.1 (b); Sec.~3.11]{Bergh1976} there is a continuous embedding \linebreak ${\left( \lspace {\overline{a}}2 ,\lspace {\overline{s}} {2p-2} \right)_{\xi, 1}\subset \lspace {\overline{r}} p}$.
Hence the choice $\XL= \lspace {\overline{r}} p $ satisfies the assumption of \Cref{thm:error_bounds}.
\\ 
The interpolation spaces  $ (\XA, K_1)_{\theta,\infty}= \left(\lspace {\overline{a}}2 , \lspace {\overline{s}} {2p-2}\right)_{\theta, \infty}$  are characterized by weighted weak $\ell^p$-spaces $\wspace \mu\nu t$ in the following manner:
\begin{align*}
\wspace \mu\nu t = (\lspace {\overline{a}}2 , \lspace{\overline{s}}{2p-2} )_{\theta,\infty} \text{ with }  \frac{1}{t}= \frac{1-\theta}{2}+\frac{\theta}{2p-2} ,\, \mu:=(\overline{a}^2 \overline{r}^{-p})^\frac{1}{2-p} ,\,\nu:=(\overline{a}\inv \overline{r})^\frac{2p}{2-p} 
\end{align*}
with equivalent quasi-norms (see \cite[Thm.~2]{F:78}). \\
The application of \Cref{thm:error_bounds} yields \Cref{thm:rates_weighted} and \Cref{rem:limiting_wei} follows from \Cref{rem:limiting_case}.
\end{proof}
%\begin{theorem}[error bounds for weighted $p$-norm penalties] 
%Let $t\in (2p-2,p)$ and $\delta, \varrho,  \alpha>0$ and $c_r\geq c_l>0$, $C_D>c_D>1$. Assume $x\in \wspace \mu\nu t  $ with $\wnorm x \mu\nu t\leq \varrho$ and $g^\delta\in \Y$ with $\| g^\delta - Ax \|_\Y\leq \delta$. Let $\xh\in R_\alpha(g^\delta).$
%There is a constant $C>0$ independent of $x$, $\delta$ and $\varrho$ such that whenever $\alpha$ satisfies either  
%\[ c_l \varrho^{-\frac{t(2-p)}{2-t}} \delta^\frac{2(2-p)}{2-t} \leq \alpha \leq c_r \varrho^{-\frac{t(2-p)}{2-t}} \delta^\frac{2(2-p)}{2-t}
%\quad\text{ or }\quad
% c_D\delta\leq \|g^\delta- A\xh \|_\Y \leq  C_D \delta \] 
%the bound 
%\[ \lnorm {x-\xh}{\overline{r}} p \leq C \varrho^\frac{\xi}{\theta} \delta^{1-\frac{\xi}{\theta}} = C \varrho^\frac{t(2-p)}{p(2-t)}  \delta^\frac{2(p-t)}{p(2-t)}.\] 
%holds true.
%\end{theorem} 
%\begin{remark} 
%\Cref{thm:rates_weighted} remains valid word by word for $p=1$ (see \cite[Thm.~4.4]{MH:20}).
%\end{remark} 
\subsubsection*{Error bounds for Besov \texorpdfstring{$0,p,p$}{}-penalties}
%Note that for $p=q <\infty$ we obtain 
%\begin{align}\label{eq:besov_weighted_identification}
% \bspace spp =\lspace {\overline{\omega}_{s,p}}p \quad\text{with equal norm for}\quad (\overline{\omega}_{s,p})_{(j,k)}= 2^{j(s+\frac{d}{2}-\frac{d}{p})} 
%\end{align}
Next we revisit Example $2.$

\begin{proof}[Proof of \Cref{thm:rates_besov_pp}]\label{pr:pp}
 Here \Cref{ass:K_1_quasi_banach} holds true by \Cref{prop:varrho_1_besov_pp}. 
The identification \cite[Thm.~2, Rem.]{F:78} for $p\neq 2$ and \cite[3.3.6.(9)]{Triebel2010} for $p=2$  yield
\begin{align}\label{eq:besov_error_norm}
\bspace 0pp = \left( \bspace {-a}22, \bspace {\tilde{s}}{\tilde{t}}{\tilde{t}} \right)_{\xi,p} = \left(\XA , K_1\right)_{\xi,p}\quad\text{with}\quad  \xi=\frac{p-1}{p}.
\end{align}
Hence the choice $\XL=\bspace 0pp$ satisfies the assumption in \Cref{thm:error_bounds}.\\
We apply \Cref{thm:error_bounds} to obtain \Cref{thm:rates_besov_pp}. 
\Cref{rem:limiting_pp} follows from \Cref{rem:limiting_case}.
\end{proof} 
Furthermore we prove the nestings given in \Cref{prop:nesting_pp}.
%For $0<s<\tilde{s}$ we set 
%\[ k_s := \left( \bspace {-a}22,  \bspace {\tilde{s}}{\tilde{t}}{\tilde{t}}\right)_{\theta,\infty} \quad\text{with}\quad \theta=\frac{p-1}{p}\frac{s+a}{a}.\]
%For $p\neq 2$ one may characterize $k_s$ by weak sequence spaces using \cite[Thm.~2]{F:78}. 
% The next proposition provides a nesting of $k_s$ for $p\neq 2$ by Besov sequence spaces.
%\begin{proposition} 
%Let $0<s<\tilde{s}$ and $t=\frac{2pa}{(2-p)s+2a}$.. 
%\begin{remunerate}
%\item For $p<2$ we have continuous embeddings
%\[ \bspace stt \subset k_s \subset \bspace s {t-\varepsilon} \infty \quad\text{for all}\quad 0< \varepsilon <t. \] 
%\item For $p>2$ we have continuous embeddings $ \bspace stt \subset k_s \subset \bspace s {2} \infty.$ 
%\end{remunerate}
%\end{proposition} 
\begin{proof}[Proof of \Cref{prop:nesting_pp}] \label{pr:nesting}
Let $\theta=\frac{p-1}{p}\frac{s+a}{a}$. Then $\frac{1}{t}=\frac{1-\theta}{2}+\frac{\theta}{\tilde{t}}$. With \cite[Thm.~2, Rem.]{F:78} and \cite[Thm.~ 3.4.1 (b)]{Bergh1976} we obtain 
\[ \bspace stt = \left( \bspace {-a}22, \bspace {\tilde{s}} {\tilde{t}}{\tilde{t}}\right)_{\theta,t} \subset \left( \bspace {-a}22, \bspace {\tilde{s}} {\tilde{t}}{\tilde{t}}\right)_{\theta,\infty} = k_s \] 
in both cases. \\
Suppose $p<2$. Then $\tilde{t}<2$ and $t \in (\tilde{t},2)$. Let $\varepsilon>0$ such that $t-\varepsilon \in (\tilde{t},2)$.
There are $s< s^\prime< \tilde{s}$ and $\theta<\theta^\prime<1$ such that $\bspace{s^\prime} {t-\varepsilon}{t-\varepsilon}= \left(\bspace {-a}22, \bspace {\tilde{s}} {\tilde{t}}{\tilde{t}} \right)_{\theta^\prime, t-\varepsilon}.$
The reiteration theorem (see \cite[Thm.~3.11.5]{Bergh1976}) yields $k_s= \left(\bspace {-a}22 , \bspace{s^\prime} {t-\varepsilon}{t-\varepsilon}\right)_{\frac{\theta}{\theta^\prime},\infty}$. From $t-\varepsilon <2$ we obtain the continuous embeddings
$\bspace {-a}22 \subset \bspace {-a}2 \infty \subset  \bspace {-a}{t-\varepsilon}\infty$ (see \cite[3.2.4(1), 3.3.1(9)]{Triebel2010}).
Together with the interpolation result 
$\bspace s {t-\varepsilon}{\infty} = (\bspace{-a}{t-\varepsilon}\infty,\bspace r {t-\varepsilon}{\infty})_{\theta,\infty}$ 
(see \cite[3.3.6 (9)]{Triebel2010}) we obtain the second inclusion using \cite[2.4.1 Rem.~4]{Triebel2010}. By \cite[3.3.1(9)]{Triebel2010}) therefore obtain the second inclusion for all $0<\epsilon <t$.\\
For $p>2$ we have $\bspace {\tilde{s}} {\tilde{t}}{\tilde{t}}\subset \bspace {\tilde{s}} {2}{\tilde{t}}$ (see \cite[3.3.1(9)]{Triebel2010}). Hence \cite[3.3.6 (9)]{Triebel2010} and \cite[2.4.1 Rem.~4]{Triebel2010} yield $k_s \subset \left(\bspace {-a}22, \bspace {\tilde{s}} {2}{\tilde{t}}\right)_{\theta,\infty}= \bspace s2\infty.$
\end{proof}

%\begin{theorem}[error bounds for $0,p,p$-penalties] 
%Let $0 < s< \frac{a}{p-1}$ and $\delta, \varrho,  \alpha>0$, $c_r\geq c_l>0$, $C_D>c_D>1$. Assume $x\in k_s$ with $\| x\|_{k_s}\leq \varrho$ and $g^\delta\in \Y$ with $\| g^\delta - Ax \|_\Y\leq \delta$. Let $\xh\in R_\alpha(g^\delta).$
%There is a constant $C>0$ independent of $x$, $\delta$ and $\varrho$ such that whenever $\alpha$ satisfies either  
%\[ c_l \varrho^{-\frac{pa}{s+a}} \delta^\frac{(2-p)s+2a}{s+a} \leq \alpha \leq c_r \varrho^{-\frac{pa}{s+a}} \delta^\frac{(2-p)s+2a}{s+a}
%\quad\text{ or }\quad
% c_D\delta\leq \|g^\delta- A\xh \|_\Y \leq  C_D \delta \] 
%the bound 
%\[ \bn {x-\xh}0 p p \leq C \varrho^\frac{a}{s+a} \delta^\frac{s}{s+a}.\] 
%holds true.
%\end{theorem} 

\subsubsection*{Error bounds for Besov \texorpdfstring{$0,2,q$}{}-penalties}
\label{sec:besov_r2q_over}
Next we treat Example $3.$ 
\begin{proof}[Proof of \Cref{thm:rates_besov_2q}]\label{pr:2q}
Due to \Cref{prop:varrho_1_besov_2q}. 
 the \Cref{ass:K_1_quasi_banach} is satisfied. 
By \cite[3.3.6.(9)]{Triebel2010} we have  
\[ \bspace 022 = \left( \bspace {-a}22, \bspace {\tilde{s}} 2 {\tilde{q}} \right)_{\xi, 2} \quad\text{with}\quad \xi=\frac{q-1}{q}.\]
Therefore the choice $\XL= \bspace 022$ satisfies the assumption on $\XL$ in \Cref{thm:error_bounds}.\\ 
Moreover for $0<s<\frac{a}{q-1}$ we have 
\[ \bspace s2\infty = \left( \bspace {-a}22,  \bspace {\tilde{s}} 2 {\tilde{q}}\right)_{\theta, 2} \quad\text{with}\quad \theta= \frac{q-1}{q}\frac{s+a}{a}. \] 
\end{proof}
Hence the application of \Cref{thm:error_bounds} yields \Cref{thm:rates_besov_2q} and  \Cref{rem:limiting_case} yields \Cref{rem:limiting_2q}.
\subsubsection*{Error bounds for the Radon transform}
Finally we turn to proof of the convergence rate result with the Radon transform as forward operator. 
\neww{
\begin{proof}[Proof of \Cref{cor:radon}]\label{pr:radon}
\begin{enumerate}
\item
Since $a< \smax$ the synthesis operator $\wav$ is a norm isomorphism $\bspace {-a}22\rightarrow \Bspace {-a}22$. Hence the operator $R\circ \wav$ satisfies \eqref{eq:XAequiA} with $\XA=\bspace {-a}22$.\\ 
The inequality \newww{$\frac{d}{p}-\frac{d}{2}\leq a$} implies $\sigma_t-\smax\leq \sigma_t \leq s$. Hence $\wav\colon \bspace stt\rightarrow \Bspace stt$ is a norm isomorphism. Let $c_1$ be the operator norm of the inverse of $\wav$. Then $f=\wav x$ with $x\in \bspace stt$ and $\bn xstt \leq c_1 \varrho.$ Let $c_2$ be the embedding constant of $\bspace stt\subset k_s$ (see \Cref{prop:nesting_pp}). Then we obtain $x\in k_s$ with $\|x\|_{k_s}\leq c_1c_2 \varrho.$\\
With $\xh$ given by \Cref{thm:rates_besov_pp} we obtain the bound \[\bn {x-\xh}0 p p \leq \tilde{C} \varrho^\frac{a}{s+a} \delta^\frac{s}{s+a}\] with a constant $\tilde{C}>0$ independent of $f$, $\delta$ and $\varrho$. 
Hence the first bound in \Cref{cor:radon} \newww{implies} 
\[ \Bn {f-\fh}0pp = \Bn {\wav(x-\xh)} 0pp\leq c_3 \bn {x-\xh} 0pp\] 
with $c_3$ the operator norm of $\wav\colon \bspace 0pp \rightarrow \Bspace 0pp$.
The bound in the $L^p$-norm for $p\leq 2$ follows from the continuity of the embedding $\Bspace 0pp\subset L^p(\Omega)$ (see \cite{Triebel2010}).
\item This follows along the lines of the proof of $1.$ using  \Cref{thm:rates_besov_2q} instead of \Cref{thm:rates_besov_pp}. 
\end{enumerate}  
\end{proof}}
%\begin{theorem}[error bounds for $0,2,q$-penalties] 
%Let $0 < s< \tilde{s}$ and $\delta, \varrho,  \alpha>0$, $c_r\geq c_l>0$, $C_D>c_D>1$. Assume $x\in \bspace s 2\infty$ with $\bn xs2\infty \leq \varrho$ and $g^\delta\in \Y$ with $\| g^\delta - Ax \|_\Y\leq \delta$. Let $\xh\in R_\alpha(g^\delta).$
%There is a constant $C>0$ independent of $x$, $\delta$ and $\varrho$ such that whenever $\alpha$ satisfies either  
%\[ c_l \varrho^{-\frac{qa}{s+a}} \delta^\frac{(2-q)s+2a}{s+a} \leq \alpha \leq c_r \varrho^{-\frac{qa}{s+a}} \delta^\frac{(2-q)s+2a}{s+a} 
%\quad\text{ or }\quad
% c_D\delta\leq \|g^\delta- A\xh \|_\Y \leq  C_D \delta \] 
%the bound 
%\[ \bn {x-\xh}0 2 2 \leq C \varrho^\frac{a}{s+a} \delta^\frac{s}{s+a}.\] 
%holds true.
%\end{theorem} 
\section{Connection to other source conditions}
\neww{In this section we return to the setting of \Cref{sec:main_minimax}  and assume only the assumptions in the first lines of \Cref{sec:main_minimax}. The aim of this section is the proof \Cref{thm:compare}.}
\subsection{A preliminary: differentiabilty of the minimal value function} \label{sec:diff}
\begin{definition}[minimal value function]\label{def:minimal_val} For $g\in \Y$ we define
 \[ \vartheta_g\colon (0,\infty)\rightarrow \mathbb{R} \quad\text{ by}\quad \vartheta_g(\alpha)= \inf\nolimits_{x\in \dom(\mathcal{R})} T_\alpha(x,g)=\frac{1}{2\alpha} \|g - A\xh \|_\Y^2 + \mathcal{R}(\xh).\] 
independent of the choice $\xh\in R_\alpha(g).$
\end{definition}
The main result of this subsection is the differentiability of the minimal value function. The approximation error $\|g-A\xh\|_\Y$ is represented by calculus rules of $\vartheta_g.$ \\
%The second clarifies the limit behavior of $\vartheta_g$ for large and small $\alpha$.
Recall that the Moreau envelope function \neww{of some function} $\mathcal{Q}\colon \,\Y \rightarrow (-\infty,\infty]$ for $\alpha>0$ is given by 
\[ \mathcal{Q}_\alpha(g) = \inf\nolimits_{y\in \Y} \left( \frac{1}{2\alpha} \| g -y  \|_\Y^2  + \mathcal{Q}(y) \right) \]
and the infimum is uniquely attained at $\prox_{\alpha \mathcal{Q}}(g)\in \Y$.
The key ingredient is the following result by T. Str\"omberg:
\begin{lemma}(see \cite[Prop. 3(iii)]{Stromberg1996}) \label{prop:stomberg}
Let $\mathcal{Q}\colon \,\Y \rightarrow (-\infty,\infty]$ be convex, proper and lower semi-continuous. The family of Moreau envelope functions $\mathcal{Q}_\alpha \colon \Y \rightarrow \mathbb{R},$  $\alpha>0$ satisfies 
\[ \frac{\partial }{\partial \alpha} \mathcal{Q}_\alpha(g) = -\frac{1}{2} \|(\nabla \mathcal{Q}_\alpha)(g)\|_\Y^2.\] 
\end{lemma}  
We apply \Cref{prop:stomberg} to the function $\mathcal{Q}$ defined in \Cref{prop:introduceQ}. Note that due to \Cref{lem:prox_of_Q} we have 
\begin{align} \label{eq:moreau_Q}
\mathcal{Q}_\alpha(g)=\frac{1}{2\alpha}\|g-\prox_{\alpha \mathcal{Q}} (g)\|_\Y^2 + \mathcal{Q}(\prox_{\alpha\mathcal{Q}}(g))=  \vartheta_g(\alpha).
\end{align}
\begin{proposition}\label{lem:diff}
Let $g\in \Y$ and $\xh\in R_\alpha(g), \alpha>0$ any selection. The function
 $\vartheta_g$ is convex, non-increasing and continuously differentiable with 
\[ \vartheta_g^\prime(\alpha)= -\frac{1}{2\alpha^2} \|g-A\xh\|_\Y^2.\] 
%\item $\alpha\vartheta_g$ is concave and continuously differentiable with 
%\[ (\alpha\vartheta_g(\alpha))^\prime (\alpha)= \mathcal{R}(\xh).\]  

\end{proposition} 
\begin{proof}
The Moreau envelope function $\mathcal{Q}_\alpha$ is convex, real valued and continuous with the Fenchel conjugate  
$(\mathcal{Q}_\alpha)^\ast=\mathcal{Q}^\ast+\frac{\alpha}{2}\|\cdot\|_\Y^2$ (see \cite[Prop. 12.15; Prop. 13.21]{Bauschke}).
The biconjugation theorem implies
\[ \vartheta_g(\alpha)=\mathcal{Q}_\alpha(g)=\left(\mathcal{Q}^\ast+\frac{\alpha}{2}\|\cdot\|_\Y\right)^\ast(g)=\sup_{v\in \Y} \left(\langle g,v\rangle - \mathcal{Q}^\ast(v)-\frac{\alpha}{2} \|v\|_\Y^2\right). \]
Hence $\vartheta_g$ is convex and non-increasing being the supremum of affine non-increasing functions. \\
By \cite[Prop. 12.29]{Bauschke} $\mathcal{Q}_\alpha$ is Fréchet differentiable with $\nabla\mathcal{Q}_\alpha= \frac{1}{\alpha}(\id_\Y -\prox_{\alpha\mathcal{Q}})$. 
\Cref{prop:stomberg} yields differentiability of $\alpha\mapsto \mathcal{Q}_\alpha(g)$ with derivative $- \frac{1}{2} \|(\nabla\mathcal{Q}_\alpha)(g) \|^2$ for all $g\in \Y.$ 
Therefore, $\vartheta_g$ is differentiable and we conclude with \Cref{lem:prox_of_Q} 
\[\vartheta_g^\prime(\alpha)= - \frac{1}{2} \|(\nabla\mathcal{Q}_\alpha)(g) \|^2= -\frac{1}{2\alpha^2} \|g-\prox_{\alpha\mathcal{Q}}(g)\|_\Y = -\frac{1}{2\alpha^2} \|g-A\xh\|_\Y. \]
Finally, $\vartheta_g^\prime$ is continuous as $\vartheta_g$ is convex and differentiable. 
\end{proof}
%\item The representation \[ \alpha\vartheta_g(\alpha) = \inf\nolimits_{x\in\XA} \left( \frac{1}{2 } \|g- Ax  \|_{\Y}^2  + \alpha \mathcal{R}(x) \right)\]  yields concavity. The identity for the derivative follows from the product rule. 
%Before we move on we state a proposition concerning the limits for $\alpha\searrow 0$ and $\alpha\rightarrow \infty$ of $\vartheta_g$. 
%\begin{proposition}\label{prop:limits_of_minimal} 
%Let $g\in \Y$. Then 
%\[ \lim_{\alpha \searrow 0} \vartheta_g(\alpha) = \mathcal{Q}(g)  \quad\text and \quad   \lim_{\alpha \rightarrow \infty} \vartheta_g(\alpha) = \inf_{z\in \XA} \mathcal{R}(z).  \] 
%In particular, $\vartheta_g$ is bounded if and only if $g \in A (\dom(\mathcal{R}))$.
%\end{proposition} 
%\begin{proof}
%Due to \cite[Prop. 12.32]{Bauschke} we have  
%$ \vartheta_g(\alpha)= \mathcal{Q}^\alpha(g) \rightarrow  \mathcal{Q}(g)$ for $\alpha \searrow 0$ and
%\[ \lim_{\alpha\rightarrow \infty}  \vartheta_g(\alpha)=\lim_{\alpha\rightarrow \infty} \mathcal{Q}^\alpha(g)=\inf_{y\in \Y}\mathcal{Q}(y)= \inf_{z\in \XA} \mathcal{R}(z).\]
%By the topological assumptions we have $-\infty< \inf_{z\in \XA} \mathcal{R}(z).$
%Hence $\vartheta_g$ is bounded if and only if $g\in \dom(\mathcal{Q})=A (\dom(\mathcal{R}))$.
%\end{proof}
\subsection{Defect function and its link to variational source conditions}
For the rest of this paper we always assume $x\in \dom(\mathcal{R})$ is $\mathcal{R}$-minimal in $A\inv (\{ Ax \})$  and $x_\alpha \in R_\alpha(Ax)$ for $\alpha>0$ is any selection of a minimizer for exact data.\\
If $A$ is injective then the minimality is trivially satisfied for all $x\in \dom(\mathcal{R})$.\\
As already mentioned we consider the \emph{defect of the Tikhonov functional}  
$\sigma_x \colon (0,\infty)\rightarrow [0,\infty)$ given by  
\[ \sigma_x(\alpha)= T_\alpha(x,Ax)-T_\alpha(x_\alpha, Ax)= \mathcal{R}(x)- \mathcal{R}(x_\alpha)- \frac{1}{2\alpha} \| Ax- Ax_\alpha\|_\Y^2.  \] 
The next proposition collects properties of the defect function. 
\begin{lemma} \label{prop:defect_props} 
\begin{remunerate}
\item $\sigma_x$ is concave, non-decreasing and continuously differentiable with $\sigma_x^\prime(\alpha)= \frac{1}{2\alpha^2} \|Ax- Ax_\alpha\|_\Y^2.$
\item We have 
$
\lim\nolimits_{\alpha \searrow 0} \sigma_x(\alpha) =0 $.
\item The function $(0,\infty)\rightarrow [0,\infty)$ given by $\alpha \mapsto \sigma_x\left( \frac{1}{\alpha}\right)$ is convex \neww{and continuous}.  
\end{remunerate}
\end{lemma} 
%\quad\text{and }  (\alpha \sigma_x(\alpha))^\prime(\alpha)= \mathcal{R}(x)- \mathcal{R}(x_\alpha). 
\begin{proof}
We have $\sigma_x(\alpha)=\mathcal{R}(x)- \vartheta_{\neww{Ax}}(\alpha)$ with the minimal value function $\vartheta_{\neww{Ax}}$ from \Cref{def:minimal_val}. Hence $1.$ follows from \Cref{lem:diff}. \Cref{app:minimal_limit} yields $2.$  because of the $\mathcal{R}$-minimality assumption on $x$. \\ 
Let $h$ be the function given in $3.$ Then $h$ is differentiable and $1.$ yields 
\[ h^\prime(\alpha)= - \frac{1}{\alpha^2} \sigma_{\neww{x}}^\prime \left(\frac{1}{\alpha} \right) = -\frac{1}{2}\|Ax-Ax_{\frac{1}{\alpha}} \|_\Y. \]
By \ref{app:monotone}.2. the function $\alpha\mapsto \|Ax - Ax_\alpha\|_\Y$ is non-decreasing. Hence $h^\prime$ is non-decreasing. Therefore $h$ is convex. \neww{Continuity follows from the first statement.}
\end{proof}
%Hence $\sigma_x$ is concave, non-decreasing and continuously differentiable with 
%\begin{align} \label{eq:lim_defect}
%\lim_{\alpha \searrow 0} \sigma_x(\alpha) = \mathcal{R}(x)- \lim_{\alpha \searrow 0}  \vartheta_g(\alpha) = \mathcal{R}(x)- \mathcal{Q}(Ax)=0
%\end{align}
%because of the $\mathcal{R}$-minimality assumption on $x$ and 
%\begin{align}\label{eq:lim_defect_infty}
%\lim_{\alpha \rightarrow \infty} \sigma_x(\alpha) = \mathcal{R}(x)- \inf_{z\in \X}\mathcal{R}(z) 
%\end{align}
Let $\alpha>0$. We write
\begin{align*}
 \sigma_x(\alpha) 
% & = \mathcal{R}(x)-  \inf\nolimits_{z\in \X} \left( \mathcal{R}(z)+\frac{1}{2\alpha} \|g-Ax\|^2\right)\\ 
 = \sup\nolimits_{z\in \X} \left( \mathcal{R}(x)- \mathcal{R}(z)-\frac{1}{2\alpha} \|Ax-Az\|^2\right)
\end{align*}
to note a similarity to the \emph{distance function} in \cite[(3.1)]{F:18} and \cite[Chapter 12]{flemming:12b}  and \cite[Chapter 3]{Hofmann2019} used to derive variational source conditions of the form \eqref{eq:vsc}.
In \cite[Prop.~4]{Hofmann2019} its shown that a variational source condition \eqref{eq:vsc} implies bounds on the defect function $\sigma_x$. 	 
The next result provides a sharp connection between bounds on the defect function and variational source conditions. We introduce two partially ordered sets of functions
\begin{align*}
\Sigma & = \left\{ \sigma \colon (0,\infty) \rightarrow [0,\infty] \, \colon \sigma \text{ is proper, non-decreasing and }\sigma\left(1/\cdot\right) \text{ is convex l.s.c.}\right\}  \\ 
\Phi & = \left\{ \phi \colon [0,\infty) \rightarrow [0,\infty) \,\colon \phi \text{ concave and upper semi-continuous} \right\}  
\end{align*}
with pointwise ordering. Here l.s.c.\ is an abbreviation for lower semi-continuous. Moreover, we consider  
the map $\mathcal{F}\colon \Sigma \rightarrow \Phi$ given by  
\begin{align}\label{eq:defF}
(\mathcal{F}(\sigma)) (t):=\inf\nolimits_{\alpha>0} \left(\sigma(\alpha)+\frac{1}{2\alpha}t \right) \quad \text{for } t\geq 0.
\end{align}
In \Cref{prop:app_F} we prove that $\mathcal{F}$ is well-defined, order preserving and bijective. The order preserving inverse $\mathcal{F}\inv\colon  \Phi \rightarrow \Sigma$ is given by  
\begin{align}\label{eq:expression_for_F_inv}
 (\mathcal{F}\inv (\phi)) (\alpha)=\sup\nolimits_{t\geq 0} \left(\phi(t)-\frac{1}{2\alpha}t \right) \quad\text{for } \alpha>0.
\end{align}

By \Cref{prop:defect_props}. we have $\sigma_x\in \Sigma$. It turns out that $\neww{\phi}=\mathcal{F}(\sigma_x)$ is the minimal function in $\Phi$ satisfying \eqref{eq:vsc}.
\begin{lemma}\label{lem:vsc}
Let $\phi\in \Phi$. Then the following statements are equivalent: 
\begin{romannum}
\item $\mathcal{F}(\sigma_x) \leq \phi$ 
\item $\sigma_x\leq \mathcal{F}\inv (\phi)$
\item $\mathcal{R}(x)-\mathcal{R}(z)\leq \phi(\|Ax-Az\|_\Y^2) \quad\text{ for all } z\in \X $.
\end{romannum} 
In particular, we always have 
\begin{align}\label{eq:bestvsc}
 \mathcal{R}(x)-\mathcal{R}(z)\leq (\mathcal{F}(\sigma_x))(\|Ax-Az\|_\Y^2) \quad\text{ for all } z\in \X 
\end{align}
%\begin{enumerate}
%\item The function 
% \[ \phi_x\colon [0,\infty) \rightarrow [0,\infty)\quad \text{given by} \quad  \phi_x(t)= \inf_{\alpha>0}\left( \sigma_x(\alpha) +\frac{1}{2\alpha}t \right)\]
%is concave, continuous and non-decreasing with $\phi_x(0)=0$. We have 
%\[ \mathcal{R}(x)-\mathcal{R}(z)\leq \phi_x(\|Ax-Az\|_\Y^2) \quad\text{for all } z\in \X.\] 
%\item If a concave and upper semi-continuous function $\phi\colon [0,\infty)\rightarrow [0,\infty)$ satisfies \eqref{eq:vsc} then $\phi_x\leq \phi$ and 
%\[ \sigma_x(\alpha)\leq \] 
%\end{enumerate}
\end{lemma} 
\begin{proof}
The equivalence of (i) and (ii) is immediate by \Cref{prop:app_F}. Next we  prove \eqref{eq:bestvsc}. To this end let $z\in \X$ and $\alpha>0$. Then 
\[ T_\alpha(x_\alpha, Ax)\leq \frac{1}{2\alpha} \|Ax-Az \|_\Y^2 + \mathcal{R}(z)\]
and $\mathcal{R}(x)= T_\alpha(x, Ax).$
We obtain 
\begin{align*}
\mathcal{R}(x)-\mathcal{R}(z) & = T_\alpha(x,Ax)-T_\alpha(x_\alpha, Ax) + T_\alpha(x_\alpha, Ax) -\mathcal{R}(z)  \\ 
& \leq \sigma_x(\alpha) + \frac{1}{2\alpha} \|Ax-Az \|_\Y^2.
\end{align*}
Taking the infimum over $\alpha$ on the right hand side yields \eqref{eq:bestvsc}. \\ 
Hence (i) implies (iii). 
Assuming (iii) we estimate 
\begin{align*}
\sigma_x(\alpha) &  \leq \phi(\|Ax-Ax_\alpha\|_\Y^2)-\frac{1}{2\alpha}\|Ax-Ax_\alpha\|_\Y^2  \\
 &\leq \sup\nolimits_{t\geq 0} \left(\phi(t)-\frac{1}{2\alpha}t \right)  =\left( \mathcal{F}\inv(\phi)\right)\neww{(\alpha)}. 
\end{align*}
\neww{Hence $\sigma_x\leq \mathcal{F}\inv(\phi)$. This yields $(i)$ as $\mathcal{F}$ is order preserving.}
\end{proof} 
\begin{remark} 
Inequality \eqref{eq:bestvsc} is sharp for $z=x_\alpha\in R_\alpha(g)$ for all $\alpha>0$. To see this note that by definition $(\mathcal{F}(\sigma_x))(t)\leq \sigma_x(\alpha)+\frac{1}{2\alpha}t$ for all $t\geq 0$ and $\alpha>0$. 
By \eqref{eq:bestvsc} we have 
\begin{align*}
\mathcal{R}(x)-\mathcal{R}(x_\alpha) & \leq \neww{ \left(\mathcal{F}(\sigma_x)\right)}\left(\|Ax-Ax_\alpha\|_\Y^2\right) \\ 
& \leq \sigma_x(\alpha)+ \frac{1}{2\alpha} \|Ax-Ax_\alpha\|_\Y^2  \\
& = \mathcal{R}(x)-\mathcal{R}(x_\alpha).
\end{align*}
\end{remark} 
%\begin{example} \label{ex:vsc_sigma_hoelder}
%Let us exploit \Cref{lem:vsc} in the case of H\"older-type index functions. 
%Let $\nu \in (0,1]$ and $c>0$. For $\sigma(\alpha):=c \alpha^\nu$ easy calculus shows that there is a constant $L_\nu$ depending only on $\nu$, such that 
%\[ (\mathcal{F}(\sigma)) (t)=\inf_{\alpha>0} \left(c \alpha^\nu+\frac{1}{2\alpha}t \right) = L_\nu c^\frac{1}{\nu+1} t^\frac{\nu}{\nu+1}. \] 
%Note that inserting $\alpha=(\frac{t}{2c})^\frac{1}{\nu+1}$ yields
%\[ (\mathcal{F}(\sigma)) (t)\leq 2^{\frac{1}{1+\nu}}c^\frac{1}{\nu+1} t^\frac{\nu}{\nu+1}\leq 2 c^\frac{1}{\nu+1} t^\frac{\nu}{\nu+1}. \]   Hence $L_\nu\leq 2.$
%\Cref{lem:vsc} yields the equivalence 
%\[\sigma_x(\alpha)\leq c \alpha^\nu   \text{ if and only if }   x \text{ satisfies } \eqref{eq:vsc} \text{ with }  \phi(t):= L_\nu c^\frac{1}{\nu+1}  t^\frac{\nu}{\nu+1}.\] 
%\end{example} 
\subsection{Link between defect function and image space approximation} 
The result of this subsection is a that $\sigma_x$ and hence also the smallest index function $\phi$ allowing for a variational source condition \eqref{eq:vsc} depends only on the net ${(\|Ax-Ax_\alpha\|_\Y)_{\alpha>0}}$.  Further we will exploit a condition when a bound $\|Ax-Ax_\alpha\|_\Y\leq \psi(\alpha)$ implies a bound on the defect function $\sigma_x$. 
\begin{lemma}
We have
\begin{align}\label{eq:delta_integrated_version}
\sigma_x(\alpha)& =\int_{0}^\alpha\frac{1}{2\beta^2} \|Ax-Ax_\beta \|_\Y^2 \, \mathrm{d}\beta  \quad\text{for all } \alpha > 0 .
\end{align}
\end{lemma} 
\begin{proof}
Let $0<\varepsilon< \alpha.$ \Cref{prop:defect_props}.1. yields
\[
\sigma_x(\alpha)-\sigma_x(\varepsilon) =\int_{\epsilon}^\alpha \sigma_x^\prime(\beta) \, \mathrm{d}\beta= \int_{\epsilon}^\alpha \frac{1}{2\beta^2} \|Ax-Ax_\alpha \|_\Y^2 \, \mathrm{d}\beta.\]
In view of \Cref{prop:defect_props}.2. the expression for $\sigma_x$ follows by taking the limit $\varepsilon\rightarrow 0$.
\end{proof}
\begin{proposition}[Image space approximation] \label{lem:image_space_vs_defect}
\begin{remunerate}
\item We have \[  \|Ax-Ax_\alpha\|_\Y\leq \sqrt{2\alpha \sigma_x(\alpha)} \quad\text{for all} \quad \alpha>0.\] 
\item Let $\psi \colon [0,\infty) \rightarrow [0,\infty)$ be continuous. Assume that there is a constant $C_\psi>0$ with 
\begin{align}\label{eq:condition_on_imagebound}
\int_0^\alpha \frac{1}{\beta} \psi(\beta)  \, \mathrm{d}\beta \leq C_\psi \psi(\alpha) \quad\text{for all } \alpha >0. 
\end{align}
Then a bound $\|Ax-Ax_\alpha\|_\Y\leq  \sqrt{2\alpha \psi(\alpha)}$ \neww{for all $\alpha>0$} implies 
$\sigma_x(\alpha)\leq C_\psi \psi(\alpha)$ \neww{for all $\alpha>0$}.
\end{remunerate} 
\end{proposition} 
\begin{proof}
\begin{remunerate}
\item By \Cref{prop:defect_props} the continuous extension of $\sigma_x$ to $[0,\infty)$ is concave. Hence the claim follows from
\[ \frac{1}{2\alpha^2} \|Ax - Ax_\alpha\|_\Y^2 = \sigma_x^\prime (\alpha)\leq \frac{1}{\alpha}\sigma_x(\alpha). \] 
\item Using \eqref{eq:delta_integrated_version} and \eqref{eq:condition_on_imagebound} we obtain 
\[ \sigma_x(\alpha)= \int_{0}^\alpha\frac{1}{2\beta^2} \|Ax-Ax_\beta \|_\Y^2 \, \mathrm{d}\beta  \leq \int_0^\alpha \frac{1}{\beta} \psi(\beta)  \, \mathrm{d}\beta \leq C_\psi \psi(\alpha).  \] 
\end{remunerate}
\end{proof}
\subsection{Equivalence theorem for H\"older-type bounds}\label{sec:equivalence}
\begin{proof}[Proof of \Cref{thm:compare}]
\item[(i)$\Rightarrow$(ii):] 
Consider the continuous function \[ \psi\colon [0,\infty)\rightarrow [0,\infty) \quad \text{ given by }  \psi(\alpha)= \frac{1}{2}c_1^2 \alpha^{2\nu-1}. \] 
Then $c_1 \alpha^\nu= \sqrt{2\alpha \psi(\alpha)}$ for all $\alpha>0.$ 
We have 
\[\int_0^\alpha \frac{1}{\beta} \psi(\beta)  \, \mathrm{d}\beta = \frac{1}{2} c_1^2 \int_0^\alpha \beta^{2\nu-2}  \, \mathrm{d}\beta= \frac{1}{2\nu-1} \psi(\alpha).\] 
Hence \eqref{eq:condition_on_imagebound} is satisfied with $C_\psi=\frac{1}{2\nu-1}$. \Cref{lem:image_space_vs_defect}. implies $ \sigma_x(\alpha)\leq \frac{c_1^2}{4\nu-2} \alpha^{2\nu-1}.$
\item[(ii)$\Rightarrow$(iii):]  
For $\sigma(\alpha):=c_2 \alpha^{2\nu-1}$ inserting $\alpha=(\frac{t}{2c})^\frac{1}{2\nu}$ yields
\[ (\mathcal{F}(\sigma)) (t) = \inf\nolimits_{\alpha>0} \left(c_2 \alpha^{2\nu-1}+\frac{1}{2\alpha}t \right) \leq 2^{\frac{1}{2\nu}}c_2^\frac{1}{2\nu} t^\frac{2\nu-1}{2\nu}\leq 2 c_2^\frac{1}{2\nu} t^\frac{2\nu-1}{2\nu}. \]  \Cref{lem:vsc} with $\phi=\mathcal{F}(\sigma)$ yields the claim. 
\item[(iii)$\Rightarrow$(i):](see also \cite[proof of Prop.~ 6]{Hofmann2019}) The first order condition \[ \xi_\alpha:= \frac{1}{\alpha} A^\ast A (x-x_\alpha)\in \partial\mathcal{R}(x_\alpha)\] provides 
\begin{align*}
\frac{1}{\alpha} \|Ax-Ax_\alpha\|_\Y^2  = \langle \xi_\alpha , x-x_\alpha \rangle  \leq \mathcal{R}(x) - \mathcal{R}(x_\alpha) \leq c_3 \|Ax-Ax_\alpha\|_\Y^\frac{2\nu-1}{\nu}.
\end{align*} 
Solving for $\|Ax-Ax_\alpha\|_\Y$ yields the claim. 
\end{proof} 
\section{Discussion and Outlook}
We close this paper by addressing some open questions and possible extensions. \\
The identification of $A\circ R_\alpha$ as a proximity mapping (see \Cref{sec:prox_mapping}) seems to be a new structural insight in convex regularization theory. It allows to apply convex analysis tools leading to interesting statements and new simple proofs (see e.g. \Cref{prop:firmly}, \Cref{lem:diff}, \Cref{prop:image_space_noisebounds}, \Cref{app:monotone}). 
So far the presented theory is limited to Hilbert space data fidelity terms. It would be interesting to generalize the arguments in \Cref{sec:mini_max} to Banach spaces $\Y$. A generalization to nonlinear operators seems even more challenging. \\
So far the presented theory is restricted to H\"older-type convergence rates. To also cover exponentially ill-posed problems it is of interest to investigate logarithmic convergence rates and source conditions. At first sight condition \eqref{eq:condition_on_imagebound} seems to fail for index functions not of H\"older-type. Thus it remains open whether an equivalence between image space approximation rates and variational source conditions remains valid for more general upper bounds.\\ 
As for approaches using variational source conditions the fastest convergence rate we are able to prove for a $p$-homogeneous penalty term is $\mathcal{O}(\frac{1}{p})$ (see \Cref{rem:limiting_wei}, \Cref{rem:limiting_pp} and \Cref{rem:limiting_2q}). It seems to be an interesting question to extend the presented approach to higher order convergence rates.\\ 
Another direction is the application to further concrete settings as in the three presented examples. An idea is to formulate a weaker version of  \Cref{ass:K_1_quasi_banach} by require a nesting $\X_{1a}\subseteq K_1 \subseteq \X_{1b}$ with quasi-Banach spaces $\X_{1a},\X_{1b}$ and try to prove a generalized  version of \Cref{thm:error_bounds}. The author believes that this approach would  cover e.g. Besov norm penalties with mixed indices $p,q$ with $p\neq 2$.
\appendix
\section{Elementary facts from regularization theory}
\begin{lemma}\label{app:monotone}
Let $g\in \Y$ and $\xh\in R_\alpha(g)$, $\alpha>0$ any selection.
\begin{remunerate}
\item The function $(0,\infty)\rightarrow \mathbb{R}$ given by $\alpha\mapsto \mathcal{R}(\xh)$ is non increasing.
\item The function $(0,\infty)\rightarrow [0,\infty)$ given by $\alpha\mapsto \|g-A\xh\|_\Y$ is non decreasing.
\item  The function $(0,\infty)\rightarrow [0,\infty)$ given by $\alpha\mapsto \frac{1}{\alpha}\|g-A\xh\|_\Y$ is non increasing.
\end{remunerate}  
\end{lemma} 
\begin{proof}
To prove $1.2.$ let $\alpha<\beta$. Set $m=\frac{1}{2}\|g-A \xh\|_\Y^2- \frac{1}{2}\|g-A\hat{x}_\beta\|_\Y^2. $ From $T_\alpha(\xh, g)\leq T_\alpha(\hat{x}_\beta,g)$ and $T_\beta ( \hat{x}_\beta,g)\leq T_\beta (\xh,g)$ we obtain  
\[ m \leq \alpha\left(\mathcal{R}(\hat{x}_\beta)-\mathcal{R}(\xh) \right)\leq \frac{\alpha}{\beta} m.\]
Hence $m\leq 0$.
$3.$ follows from \Cref{lem:diff}. 
\end{proof}
\begin{lemma}\label{app:minimal_limit}
Let $x\in \XA$ and $x_\alpha\in R_\alpha(Ax), \alpha>0$ any selection.  Then \[ \lim\nolimits_{\alpha\searrow 0} \left( \frac{1}{2\alpha}\|Ax-Ax_\alpha\|_\Y^2 + \mathcal{R}(x_\alpha) \right) = \inf \{ \mathcal{R}(z) \colon z\in \XA \text{ with } Az=Ax\}.  \]  
\end{lemma} 
\begin{proof}
Due to \eqref{eq:moreau_Q} and \cite[Prop. 12.32]{Bauschke} we have  
\[ \frac{1}{2\alpha}\|Ax-Ax_\alpha\|_\Y^2 + \mathcal{R}(x_\alpha)= \mathcal{Q}_\alpha(Ax) \rightarrow  \mathcal{Q}(Ax)  \quad\text{ for } \alpha \searrow 0  \]
with $\mathcal{Q}$ defined in \Cref{prop:introduceQ} and $\mathcal{Q}_\alpha$ its Moreau envelope (see \Cref{sec:diff}).
\end{proof}
\section{Properties of Banach spaces}
\begin{proposition} \label{app:dual}
\begin{remunerate}
\item Let $p\in [1,\infty)$ and $\omega=(\omega_j)_{j\in\Lambda}$ a sequence of positive reals. Let $p^\prime\in (1,\infty]$ with $\frac{1}{p}+\frac{1}{p^\prime}=1$. Then the pairing 
\[ \langle\cdot , \cdot \rangle \colon \lspace {\omega\inv} {p^\prime} \times \lspace \omega p\rightarrow \mathbb{R} \quad\text{given by } \langle \xi, x\rangle=\sum\nolimits_{j\in \Lambda} \xi_j x_j \] 
is well defined and gives rise to an isometric isomorphism $(\lspace \omega p)^\prime\cong \lspace {\omega\inv}{ p^\prime}$.
\item Let $p,q\in [1,\infty)$ and $s\in \mathbb{R}$. Then the pairing 
\[ \langle\cdot , \cdot \rangle \colon \bspace {-s} {p^\prime} {q^\prime} \times \bspace s p q \rightarrow \mathbb{R} \quad\text{given by } \langle \xi, x\rangle=\sum\nolimits_{(j,k)\in \Lambda} \xi_{j,k} x_{j,k} \] 
is well defined and gives rise to an isometric isomorphism $(\bspace spq )^\prime\cong \bspace {-s} {p^\prime} {q^\prime}$. (see \cite[2.11.2 (1)]{Triebel2010})
\end{remunerate}
\end{proposition} 

\begin{proposition} \cite[Lem. 8.21.]{Scherzer_etal:09} \label{lem:boundedness_linear_fct_range}
Let $A\colon \X\rightarrow \Y$ be a bounded linear operator between Banach spaces and $\xi\in \X^\prime$. The following statements are equivalent:  
\begin{romannum}
\item There exists a constant $c\geq 0$ such that $\langle \xi , x \rangle \leq c \|Ax \|_\Y$ for all $x\in \X.$
\item There exists $\omega\in \Y^\prime$ with $\|\omega\|_{\Y^\prime}\leq c$ and $A^\ast \omega = \xi.$
\end{romannum}
\end{proposition} 

%\begin{proposition} \label{app:compactness_of_unitball}
%Let $a>0$, $r\geq 0$ and $q\in [1,\infty]$. The unit ball \[ B:=\{ x\in \bspace r2q \colon \bn x r2q\leq 1 \}\]  is norm compact in  $\bspace {-a}22$. 
%\end{proposition} 
%\begin{proof}
%For $n\in \mathbb{N}$ consider the projection $P_n \colon \bspace r2q \rightarrow \bspace {-a}22$ given by $(P_n x)_{(j,k)}= x_{(j,k)}$ if $j< n$ and $(P_n x)_{(j,k)}=0$ for $j\geq n$.
%By \cite[Ex.~3.8(b)]{WSH:20} we obtain a Jackson-type inequality 
%\[ \bn{ x-  P_n x} {-a} 22 \leq c 2^{-n(a+r)} \bn x r2q \] with a constant $c>0$ depending only on $a,r$ and $q$. This shows $\bspace r2q$ embeds continuously into $\bspace {-a}22$ and that finite range operator $P_n$ converges strongly to the embedding operator. Hence the embedding is compact. Therefore $B$ is precompact. It remains to show that $B$ is closed. To this end let $(x^{(k)})_k\subset B$ be a sequence that converges to some $x\in \bspace {-a}22$ in the $\bn \cdot {-a} 22$-norm. In particular we obtain $x_\lambda = \lim_k x_\lambda^{(k)}$ for all $\lambda\in \Lambda$. As $P_n$ has finite range this implies 
%\[  \bn {P_n x} r2q = \lim_k \bn {P_n x^{(k)}} r2q  \leq \limsup_k \bn {x^{(k)}} r2q \leq 1 \] 
% for all $n\in \mathbb{N}$. Therefore, $x\in B.$
%\end{proof}
\section{Index function calculus} ~\\
Let $\Gamma:= \left\{f\colon \mathbb{R}\rightarrow (-\infty,\infty] \colon f \text{ is proper, convex and lower semi-continuous} \right\}$. 
\begin{lemma} \label{prop:app_fenchel}
Suppose $f\in \Gamma$. Then
\begin{remunerate}
\item $f$ is positive with $\dom(f)\subseteq (-\infty,0]$ if and only if $f^\ast|_{[0,\infty)}\leq 0$ . 
\item $f$ is non-decreasing if and only if  $\dom(f^\ast)\subseteq [0,\infty)$.
\end{remunerate}
%\item $\dom(f)\subseteq  [0,\infty)$ and $f\leq 0$ on $[0,\infty)$.  
%\item $\dom(f)\subseteq (-\infty,0]$ and $f$ is non-decreasing and positive. 
%\end{thmlist}
\end{lemma} 
\begin{proof}
\begin{remunerate}
\item $f$ is positive with $\dom(f)\subseteq (-\infty,0]$ if and only if $\chi_{(-\infty,0]}\leq f$. $f^\ast|_{[0,\infty)}\leq 0$ if and only if $f^\ast\leq \chi_{[0,\infty)}$.  Hence the claim follows from $\chi_{[0,\infty)}^\ast=\chi_{(-\infty,0]}$.
\item Suppose $f$ is non-decreasing and let $t<0$. Let $\beta_0\in \dom(f)$. Then \[ \beta t - f(\beta) \geq \beta t - f(\beta_0) \quad\text{for all }  \beta\leq \beta_0.\] As $\beta t - f^\ast(\beta_0)\longrightarrow \infty$  for $\beta\rightarrow -\infty$ this shows $f^\ast(t)=\sup_{\beta\in\mathbb{R}}\beta t - f(\beta)=\infty$. Hence $\dom(f^\ast)\subseteq [0,\infty).$\\ Vice versa assume $\dom(f^\ast)\subseteq [0,\infty)$. Then $f(\beta)=\sup_{t\geq 0} t\beta  - f^\ast(t)$ is non-decreasing as a supremum over non-decreasing functions.    
\end{remunerate}
\end{proof}
\begin{lemma} \label{prop:app_F}
The map $\mathcal{F}$ defined in \eqref{eq:defF} is well-defined, order preserving and bijective. The expression \eqref{eq:expression_for_F_inv} holds true.
\end{lemma} 
\begin{proof}
We define the following sets
\begin{align*} 
\Gamma_1 &= \left\{ f\in \Gamma \colon f  \text{ is non-decreasing with } \dom(f)\subseteq (-\infty,0] \right\} \\  
\Gamma_2 &= \left\{ f\in \Gamma \colon \dom(f)\subseteq [0,\infty) \text{ and } f|_{[0,\infty)} \leq 0 \right\}
\end{align*}
By \Cref{prop:app_fenchel} the Fenchel conjugation $ ^\ast\colon \Gamma_1 \rightarrow \Gamma_2$ is an order reversing bijection and its inverse is given by the Fenchel conjugation $ ^\ast\colon \Gamma_2 \rightarrow \Gamma_1$. We will construct bijections $\mathcal{G}_1 \colon \Sigma\rightarrow \Gamma_1$ and $\mathcal{G}_2 \colon \Gamma_2 \rightarrow \Phi$, such that $\mathcal{F}=\mathcal{G}_2 \circ \, ^\ast \circ \mathcal{G}_1$. \\
Let $\sigma\in \Sigma.$ Then we define
\begin{align*}
 f_\sigma \colon \mathbb{R}\rightarrow [0,\infty]\text{ by } f_\sigma(\beta)=\begin{cases} 
\sigma\left( -\frac{1}{2\beta}\right) &\text{if } \beta < 0 \\ 
\lim_{\alpha\rightarrow \infty} \sigma\left( \alpha \right) &\text{if } \beta=0 \\ 
\infty &\text{if } \beta > 0  
\end{cases}.
\end{align*}
Then $f_\sigma$ is proper, non-decreasing and $\dom(f_\sigma)\subset (-\infty, 0].$
Convexity \neww{and lower semi-continuity}  of $\sigma(\frac{1}{\cdot})$ yields convexity \neww{and lower semi-continuity} of $f_\sigma$ on $(-\infty,0)$. We have 
\[ f_\sigma(0)=\lim_{\beta \nearrow 0} \sigma\left(-\frac{1}{2\beta}\right) = \liminf_{\beta\rightarrow 0} f_\sigma(\beta).\] 
Hence $f_\sigma$ is convex and lower semi-continuous.\\ 
It is easy to see that $\mathcal{G}_1 \colon \Sigma\rightarrow \Gamma_1$ given by $\sigma\mapsto f_\sigma$ is a order preserving bijection. Its inverse is given by $(\mathcal{G}_1\inv (f))(\alpha)= f\left(-\frac{1}{2\alpha}\right).$\\ 
Moreover, the map $\Gamma_2\rightarrow \Phi$ given by $g\mapsto -(g|_{[0,\infty)})$ is well defined, bijective and order reversing. Its inverse is given by $\phi\mapsto g_\phi$ with   
\begin{align*}
 g_\phi \colon \mathbb{R}\rightarrow (-\infty,\infty]\text{ given by } g_\phi(t)=\begin{cases} 
-\phi\left( t \right) &\text{if } t \geq  0 \\ 
\infty &\text{if } t < 0  
\end{cases}.
\end{align*}
If $\sigma\in\Sigma$ and $t\geq 0$
then 
\[ \lim\nolimits_{\beta \nearrow 0 }\beta t - f_\sigma(\beta) =- \lim\nolimits_{\beta\nearrow 0 }f_\sigma(\beta) = -f_\sigma(0). \] 
Hence 
\begin{align*}
(\mathcal{F}(\sigma))(t)& =\inf\nolimits_{\alpha>0} \left(\sigma(\alpha)+\frac{1}{2\alpha}t \right) \\ 
& = \inf\nolimits_{\beta<0} f_\sigma(\beta)- \beta t \\
& = \inf\nolimits_{\beta\leq 0} f_\sigma(\beta)- \beta t \\
& = - f_\sigma^\ast(t)=   ((\mathcal{G}_2 \circ \, ^\ast \circ \mathcal{G}_1)(\sigma))(t)
\end{align*}
This shows $\mathcal{F}=\mathcal{G}_2 \circ \, ^\ast \circ \mathcal{G}_1$. Therefore $\mathcal{F}$ is an order preserving bijection. It remains to compute $\mathcal{F}\inv= \mathcal{G}_1\inv \circ  \, ^\ast \circ \mathcal{G}_2\inv$. If $\phi\in \Phi$ and $\alpha>0$, then 
\[ (\mathcal{F}\inv(\phi))(\alpha)= g_\phi^\ast\left(-\frac{1}{2\alpha}\right)= \sup\nolimits_{t\geq 0} \left( -g_\phi(t) - \frac{1}{2\alpha}t \right)= \sup\nolimits_{t\geq 0} \left( \phi(t) - \frac{1}{2\alpha}t \right). \]
\end{proof} 
\section*{Acknowledgments}
I would like to thank Thorsten Hohage and Benjamin Sprung for fruitful discussions, Matthew Tam and Russell Luke for their support concerning convex analysis topics and Thomas Str\"omberg for \Cref{prop:stomberg}.\\
Financial support by \neww{Deutsche Forschungsgemeinschaft (DFG, German Research \linebreak Foundation) through grant RTG 2088 - B01 }is gratefully acknowledged.
\bibliographystyle{siamplain}
\bibliography{lit}
\end{document}